\documentclass[12pt]{article}
\usepackage{amsmath,amssymb,amsthm,bbm,color,xcolor}
\usepackage[round]{natbib}
\usepackage[margin=1in]{geometry}
\usepackage{graphicx}
\usepackage{mathrsfs}
\RequirePackage[colorlinks,citecolor={blue!60!black},urlcolor={blue!70!black},linkcolor={red!60!black},breaklinks,hypertexnames=false]{hyperref}

\newtheorem{theorem}{Theorem}
\newtheorem{lemma}[theorem]{Lemma}
\newtheorem{proposition}[theorem]{Proposition}

\newtheorem{corollary}[theorem]{Corollary}

\newtheorem{example}{Example}

\newtheorem{innercustomexample}{Example}
\newenvironment{customexample}[1]
  {\renewcommand\theinnercustomexample{#1}\innercustomexample}
  {\endinnercustomexample}

\newcommand{\R}{\mathbb{R}}

\newcommand{\EE}[1]{\mathbb{E}\left[{#1}\right]}
\newcommand{\EEst}[2]{\mathbb{E}\left[{#1}\  \middle| \ {#2}\right]}
\newcommand{\Ep}[2]{\mathbb{E}_{{#1}}\left[{#2}\right]}

\newcommand{\PP}[1]{\mathbb{P}\left\{{#1}\right\}}

\newcommand{\Ppst}[3]{\mathbb{P}_{{#1}}\left\{{#2}\  \middle| \ {#3}\right\}}
\newcommand{\Pp}[2]{\mathbb{P}_{{#1}}\left\{{#2}\right\}}

\newcommand{\One}[1]{{\mathbbm{1}}\left\{{#1}\right\}}
\newcommand{\one}[1]{{\mathbbm{1}}_{{#1}}}
\newcommand{\iidsim}{\stackrel{\textnormal{iid}}{\sim}}

  \newcommand\independent{\protect\mathpalette{\protect\independenT}{\perp}}
\def\independenT#1#2{\mathrel{\rlap{$#1#2$}\mkern2mu{#1#2}}}

\newcommand{\alg}{\mathcal{A}}
\newcommand{\cZ}{\mathcal{Z}}
\newcommand{\cW}{\mathcal{W}}
\newcommand{\cH}{\mathcal{H}}

\newcommand{\cI}{\mathcal{I}}
\newcommand{\dtv}{\textnormal{d}_{\textnormal{TV}}}
\newcommand{\bfa}{\mathbf{a}}

\title{Algorithmic stability via ensembling}
\author{Rina Foygel Barber\thanks{University of Chicago, \texttt{rina@uchicago.edu}} \and Richard J.\ Samworth\thanks{University of Cambridge, \texttt{r.samworth@statslab.cam.ac.uk}}}
\begin{document}
\maketitle

\begin{abstract}
    Algorithmic stability refers to the property of an algorithm being insensitive to perturbations of the input data, where the type of perturbation may vary depending on the setting. In this work, we develop a general framework to quantify the extent to which any ensembling strategy defined via averaging can yield stability guarantees for any type of data perturbation.
    Our main theoretical result is a guarantee on the stability of this ensembled algorithm, given in terms of the norm of a certain covariance operator that describes the ensembling process.  
    We show how our general framework yields interpretable and intuitive insights in several examples of perturbations of practical interest, and provides much sharper guarantees than those obtained from privacy considerations.
\end{abstract}

\section{Introduction}\label{sec:intro}

In modern data analysis, there are many factors that play a role in choosing an algorithm for learning from data: we may wish to select an algorithm based on features such as accuracy (in recovering some underlying model or parameter), computational cost, privacy, robustness to outliers or to model errors, and many more. In this work, we are interested in the property of \emph{algorithmic stability}, which measures the extent to which the output of an algorithm is sensitive to small perturbations of the training data; in particular, in the statistics literature, stability is commonly defined by considering the change in the output when deleting, or resampling, small amounts of the training data, but other types of perturbations have also been considered. Stability is widely considered to be a fundamental desirable property of learning algorithms \citep{meinshausen2010stability,yu2013stability,shah2013variable,murdoch2019definitions}, and is closely connected with other statistical aims and properties such as generalization, learnability, reproducibility, and robustness \citep{mukherjee2006learning,shalev2010learnability,yu2020veridical}. 

Certain types of algorithms are provably stable---for instance, algorithms that are defined by strongly convex regularized optimization problems, where the regularization enforces low sensitivity to perturbations \citep{bousquet2002stability,wibisono2009sufficient}. While the algorithms used in practice in many applications are often more complex, empirically they may appear to exhibit stability, but it is known to be impossible to empirically certify stability properties when data are limited \citep{kim2021black,luo2024algorithmic}. 

Consequently, a recent question of interest is whether it is possible to introduce a post-hoc correction to an arbitrary base algorithm, enforcing stability without losing its favorable performance. The recent work of \citet{soloff2024bagging,soloff2024stability} shows that \emph{bagging}\footnote{This term, which is an abbreviation of `bootstrap aggregating', was coined by \citet{breiman1996bagging}.} (that is, repeatedly resampling the data, and then averaging the algorithm's output) ensures a guarantee of stability in the sense of low sensitivity to deleting small amounts of data. In this work, we develop a framework to study this phenomenon more generally: given any base algorithm, and any desired notion of stability (i.e., stability with respect to some sort of perturbation), can we use some form of ensembling in order to ensure that this stability property will hold?

\subsection{Algorithmic stability and algorithm ensembling}\label{sec:intro_stability_ensembling}
Let $\alg$ be a map that inputs data $z\in\cZ$ and returns an output $\alg(z)\in[0,1]$. We would like to assess whether $\alg$ is sensitive to perturbations of the input data.  To motivate this question, we now introduce several key examples that we will study throughout the paper. First, we might consider stability with respect to dropping a single data point:
\begin{example}[Deleting one data point]\label{example1}
Let $z = (z_1,\dots,z_n)$ be a collection of training data points, and write $z_{-i} = (z_1,\dots,z_{i-1},z_{i+1},\dots,z_n)$ to denote the same data set when the $i$th data point has been deleted. Then the stability of $\alg$, with respect to deleting a data point from the training dataset $z$, can be summarized by computing
\begin{equation}
\label{Eq:DeletionStability}
\frac{1}{n}\sum_{i=1}^n \big(\alg(z) - \alg(z_{-i})\big)^2,
\end{equation}
which is the average change in the output of $\alg$ when a single data point is deleted at random from $z$.
\end{example}
\noindent In this example, even though the data vector $z$ is fixed (and might be chosen adversarially), the quantity~\eqref{Eq:DeletionStability} nonetheless captures an `average' rather than `worst-case' notion of stability, since the data value $z_i$ being deleted is being chosen at random from $i\in[n]$.

Second, it is common to consider stability with respect to data corruption:
\begin{example}[Corrupting part of the data]\label{example2} For a random data vector $Z\in\R^n$, we can define a perturbed version of~$Z$ by replacing a randomly chosen subset of data point with contaminated or corrupted values. Specifically, suppose that we select a subset $S\subseteq[n]$ at random, with $|S|=s$, and each entry $Z_i$ for $i\in S$ is replaced by some corrupted value $\zeta_i$: we write
\[Z\circ \mathbf{1}_{S^c} + \zeta\circ \mathbf{1}_S\]
to denote the corrupted data vector, where $\mathbf{1}_S\in\{0,1\}^n$ is the indicator vector for the set $S$ (and similarly for $\mathbf{1}_{S^c}$), and $\circ$ denotes elementwise product. Then the quantity
\[
\frac{1}{\binom{n}{s}}\sum_{\substack{S\subseteq[n]\\|S|=s}}\Ep{Z\sim \pi,  \zeta\sim \nu}{\Big(\alg(Z) - \alg\big(Z\circ \mathbf{1}_{S^c} + \zeta\circ \mathbf{1}_S\big)\Big)^2}
\]
characterizes the stability of $\alg$ in this setting, with respect to randomly sampled data drawn as $Z\sim\pi$, and corruptions drawn as $\zeta\sim\nu$, with the indices $S$ of the corrupted entries chosen uniformly at random.
\end{example}

It is natural to ask how one might make an algorithm more stable to this type of perturbation, and bagging provides one common approach.  Here we present two common variants \citep{breiman1996bagging,andonova2002simple}:
\begin{customexample}{A}[Bagging]\label{exampleA}
    For constructing the bagged version of an algorithm $\alg$, we may average the output of $\alg$ with respect to drawing a random subset of size~$m$, sampled uniformly without replacement from $[n]$ (also called `subbagging'):
    \[\alg_{\textnormal{bag}}(z) =  \frac{1}{n!/(n-m)!}\sum_{\substack{i_1,\dots,i_m\in[n]\\ \textnormal{s.t.\ $i_k\neq i_\ell$ $\forall$ $k\neq \ell$}}}\alg\big((z_{i_1},\dots,z_{i_m})\big).\]
        Alternatively, we may sample $m$ indices uniformly with replacement (`the $m$ out of $n$ bootstrap'): 
    \[\alg_{\textnormal{bag}}(z) =  \frac{1}{n^m}\sum_{i_1,\dots,i_m\in[n]} \alg\big((z_{i_1},\dots,z_{i_m})\big).\]
\end{customexample}

Another common technique for increasing stability is to `smooth' the algorithm by averaging over noisy versions of the data:
\begin{customexample}{B}[Adding noise]\label{exampleB}  Given data $z\in\R^n$, we define an ensembled algorithm by averaging over noisy versions of the data, $z+\xi$, where $\xi\in\R^n$ is a vector of noise, with entries drawn i.i.d.\ from some density $h$ on $\R$. This leads to a smoothed algorithm:
\[\alg_{\textnormal{smth}}(z) = \Ep{\xi_i\iidsim h}{\alg(z+\xi)}.\]
\end{customexample}
The goal of this paper is to study how ensembling strategies, such as bagging or averaging over added noise as in Examples~\ref{exampleA} and~\ref{exampleB}, relate to the stability properties of the resulting ensembled algorithms, such as stability with respect to deleting or corrupting data as in Examples~\ref{example1} and~\ref{example2}. As we will see in the overview of related work below, some existing literature has established this type of connection for specific examples; in this paper, our aim is to build a broader and more general understanding of the connection between ensembling and stability.

\subsection{Related work}

Stability has been extensively studied in the statistics and learning theory literature.  It is often viewed as a desirable algorithmic property in its own right, connected to qualitative goals such as interpretability or trustworthiness \citep{breiman1996heuristics,yu2013stability}.
\citet{yu2020veridical} discuss the necessity of stability with respect to many types of perturbations or choices made by the analyst, including variability of the data, randomness inherent to the procedure, and decisions made in constructing the model or in data preprocessing.

From a theoretical standpoint, the most commonly studied notion of stability is that of stability with respect to deleting---or resampling---a single data point or a small fraction of the data (as in Example~\ref{example1}); see \citet{bousquet2002stability} for background. This type of stability condition is known to be sufficient for establishing properties such as generalization \citep{bousquet2002stability,poggio2004general}, learnability \citep{shalev2010learnability}, asymptotic convergence of cross-validation-based estimates of risk \citep{austern2020asymptotics,bayle2020cross}, avoidance of sample splitting \citep{lundborg2024projected}, and validity of cross-validation-based predictive inference \citep{barber2021predictive,amann2023assumption}. \citet{chakraborty2026stability} examine the tradeoff between stability and accuracy via a constrained minimax framework.

While stability with respect to data deletion or resampling has been extensively studied, a broad range of different notions of stability (under various names) have been proposed in statistics and in related fields. Here we mention several examples. The field of robust statistics examines the sensitivity of statistical procedures to data contamination or outliers; see \citet{huber2011robust} for background. In Bayesian statistics, \emph{prior sensitivity analysis} refers to examining whether the conclusions (the posterior) are stable with respect to various choices of the prior \citep{gelman1995bayesian}. In optimization, it is common to study stability with respect to the choice of initialization point, or with respect to a perturbation to the optimization problem---for instance, \citet{poliquin1998tilt} study the \emph{tilt stability} of a minimization problem, i.e., the perturbation to the solution that results from adding a random linear term (a `tilt') to the function being minimized. 

There are many statistical tools and procedures that have been developed to improve stability (and other properties) empirically, most notably bagging \citep{breiman1996bagging} and subbagging \citep{andonova2002simple}, as described in Example~\ref{exampleA}. Bagging was introduced as a tool for constructing stable models by \citet{breiman1996bagging,breiman1996heuristics}, and plays a key role in statistics and machine learning, underlying many common procedures, such as random forests \citep{breiman2001random}.  In certain situations, it is known to improve performance by reducing variance \citep{buhlmann2002analyzing,hall2005properties,biau2010rate,samworth2012optimal}.  As mentioned earlier, the work of \citet{soloff2024bagging,soloff2024stability} establishes that bagging any base algorithm leads to a particular type of stability guarantee, namely, stability with respect to deleting one data point (Example~\ref{example1}); we will give details on these results in Section~\ref{sec:example1+A}.

\subsection{Our contributions}

In Section~\ref{sec:framework}, we introduce definitions that quantify the stability of an algorithm to an arbitrary notion of data perturbation, and construct an ensembled version of the algorithm that averages outputs over data drawn using a Markov kernel (or `channel'), generalizing the various special cases introduced in the examples above. 

Our main theoretical results are presented in Section~\ref{sec:main_results}, providing an upper bound on the stability parameter of the ensembled algorithm, given in terms of a certain operator norm of the covariance operator derived from a kernel associated with the channel. (This bound is presented for algorithms taking values in the unit interval $[0,1]$ but can be extended via Grothendieck's inequality to more general settings where the algorithm takes values in a real, separable Hilbert space, e.g., algorithms that return a fitted function---see Appendix~\ref{sec:HSExtension}.)
Section~\ref{Sec:Examples} is devoted to several examples that showcase the utility of our general framework.  In each case, we show how our theory yields interpretable and intuitive bounds on the stability parameter.

In general, one way of ensuring stability is via privacy: a sufficiently noisy channel will make the differences between original and perturbed data undetectable.  In Section~\ref{Sec:Privacy}, we explore the connections between stability and privacy in detail, reaching the conclusion that our theory provides much tighter bounds on stability than those implied via privacy.  We end with a discussion of related ideas and possible extensions in Section~\ref{Sec:Discussion}.  Most proofs are deferred to Appendix~\ref{sec:proofs_appendix}.

\section{Framework}\label{sec:framework}

In this section, we will develop a unified framework and notation for describing stability with respect to perturbations (for instance, perturbations such as deleting data or corrupting data, as described in Examples~\ref{example1} and~\ref{example2} above), and for developing a generic notion of ensembled algorithms (for example, bagging as in Example~\ref{exampleA}, or averaging with respect to added noise as in Example~\ref{exampleB}).  

\subsection{A general definition of stability}
\label{Subsec:GeneralDef}

We begin by introducing notation that will allow us to generalize to an arbitrary notion of data perturbation.  Assume the data $z$ lies in some measurable space $\cZ$ and let $P$ be a distribution on $\cZ\times\cZ$. In a random pair $(Z,Z')\sim P$, $Z$ represents the original unperturbed data, while $Z'$ is its perturbed version. Our goal is to guarantee that our algorithm returns similar outputs when trained on $Z$ versus on $Z'$, when we sample $(Z,Z')\sim P$.

To return to our motivating examples:
\begin{customexample}{\ref{example1}}[Deleting one data point]
    Given a data vector $z=(z_1,\dots,z_n)$, define 
    \[P = \frac{1}{n}\sum_{i=1}^n \delta_{(z,z_{-i})}.\]
    This is the uniform distribution over all possible pairs $(z,z_{-i})$---that is, the original data $z$, and the perturbed data $z_{-i}$ obtained by deleting one data point $i\in[n]$, with $i\in[n]$ chosen at random, as described in Example~\ref{example1}.
\end{customexample}

\begin{customexample}{\ref{example2}}[Corrupting part of the data]
    For the corrupted data model described in Example~\ref{example2} above, we can define $P$ as the distribution of
    \[\big(Z,Z\circ \mathbf{1}_{S^c} + \zeta\circ \mathbf{1}_S\big)\]
    that is induced by sampling $Z,\zeta,S$ independently as
    \[Z\sim \pi,\quad \zeta\sim\nu, \quad S\sim\textnormal{Unif}\left(\binom{[n]}{s}\right),\]
    where $\binom{[n]}{s}$ denotes the collection of all subsets of $[n]$ with cardinality $s$.
\end{customexample}

Given a joint distribution $P$, we define the \emph{stability parameter}
\[\beta^2_P(\alg) =\Ep{(Z,Z')\sim P}{\big(\alg(Z) - \alg(Z')\big)^2},\]
where the expectation is taken with respect to randomly sampling $(Z,Z')\sim P$, i.e., the original data $Z$ and its perturbed version $Z'$.
Intuitively, an algorithm $\alg$ is \emph{stable} to perturbations if this quantity is small. But depending on the nature of the algorithm, this may not immediately be the case. The aim of ensembling is to \emph{stabilize} the algorithm $\alg$: we aim to modify the algorithm in such a way that this stability parameter is guaranteed to be small.

\subsection{Defining an ensembled algorithm}
\label{Subsec:Ensembled}

As discussed in Section~\ref{sec:intro_stability_ensembling}, to help ensure that an algorithm $\alg$ achieves some notion of stability, it is common to modify $\alg$ by taking an average over perturbations of the data. We have seen two specific mechanisms that are commonly used: subsampling data (Example~\ref{exampleA}) and adding noise to data (Example~\ref{exampleB}).

To work within a more general framework that captures these existing mechanisms as well as other possible notions of ensembling, we now introduce some additional notation. Let $Q$ denote a Markov kernel, with $Q(\cdot \mid z)$ defining a distribution on $\cZ$, for each $z\in\cZ$. 
Given a choice of~$Q$, we define
\begin{equation}\label{eqn:define_alg_nu_z}\alg_Q(z) = \Ep{Z'\sim Q(\cdot\mid z)}{\alg(Z')}.\end{equation}
We can think of this as an ensembled version of the algorithm $\alg$, obtained by averaging over the draw of the perturbed data $Z'\sim Q(\cdot \mid z)$.\footnote{In practice, this expected value would likely be estimated with Monte Carlo sampling, by averaging over a large finite number of draws from $Q(\cdot\mid z)$, but we do not examine this extension here.} Borrowing terminology from the fields of information theory and differential privacy, we can refer to $Q$ as a `channel' that converts data $z$ into a randomly sampled output drawn from $Q(\cdot\mid z)$. Since $Q$ defines an ensembled version of the algorithm $\alg$, we will call it the \emph{ensembling channel}.

To return to our earlier examples:
\begin{customexample}{\ref{exampleA}}[Bagging]
    Let $z = (z_1,\dots,z_n)$ denote the data. For subbagging (sampling uniformly without replacement), we define    
    \[Q(\cdot\mid z) = \frac{1}{n!/(n-m)!}\sum_{\substack{i_1,\dots,i_m\in[n]\\ \textnormal{s.t.\ $i_k\neq i_\ell$ $\forall$ $k\neq \ell$}}}\delta_{(z_{i_1},\dots,z_{i_m})}.\]
    For bootstrapping (sampling uniformly with replacement), we instead have
     \[Q(\cdot\mid z) = \frac{1}{n^m}\sum_{i_1,\dots,i_m\in[n]}\delta_{(z_{i_1},\dots,z_{i_m})}.\]
\end{customexample}

\begin{customexample}{\ref{exampleB}}[Adding noise]
For adding noise $\xi$ with entries sampled i.i.d.\ from some density $h$ on $\R$, we define the ensembling channel $Q$ via the conditional density
\[Q(w\mid z) = \prod_{i=1}^n h(w_i - z_i).\]
\end{customexample}

\section{Main results}\label{sec:main_results}

The aim of our theoretical analysis is to derive a simple and assumption-lean bound on the stability parameter for the ensembled algorithm $\alg_Q$, i.e., $\beta^2_P(\alg_Q)$, working in the general framework defined above. 
Intuitively, in order for this quantity to be small, we need to choose the channel $Q$ to be sufficiently noisy so as to ``mask'' the effect of the perturbation, i.e., the difference between data $Z$ and its perturbed version $Z'$ when $(Z,Z')\sim P$.

As mentioned above, such bounds have already been established in the specific setting of deleting a small portion of the data: it is known that bagging (i.e., the ensembling mechanism defined in Example~\ref{exampleA}) leads to stability with respect to deleting data (i.e., the notion of stability defined in Example~\ref{example1}). However, in this work our aim is to develop a general theory that will enable stability bounds to be proved more broadly, and will also offer new insights on the reason that bagging offers stability in the setting of Example~\ref{example1}.

\subsection{Preliminaries}
We will assume that the ensembling channel $Q$ is chosen such that $Q(\cdot\mid z)$ has a density with respect to some $\sigma$-finite base measure $\mu$ on $\cZ$, for every $z\in\cZ$ (for instance, the counting measure for the discrete setting of Example~\ref{exampleA}, or the Lebesgue measure on $\R^n$ for Example~\ref{exampleB}). Overloading notation, we will use $Q(\cdot\mid z)$ to denote the density or the distribution, as appropriate.

First we define a kernel
\[K_{P,Q}(z,z') = \Ep{(Z_0,Z_1)\sim P}{\big(Q(z\mid Z_0) - Q(z\mid Z_1)\big)\cdot \big(Q(z'\mid Z_0) - Q(z'\mid Z_1)\big)}.\]
Note that $K_{P,Q}$ is positive semidefinite by construction.  Let $T_{P,Q}$ be the associated operator: for any bounded measurable function $f:\cZ\to\R$, define $T_{P,Q}f:\cZ\to\R$ as
\[
[T_{P,Q}f](z) = \int_{\cZ} K_{P,Q}(z,z') f(z')\;\mathsf{d}\mu(z').
\]
The following lemma gives an interpretation of this operator:
\begin{lemma}\label{lem:operator_avg}
    Let $f,g:\cZ\to\R$ be bounded functions. Define their ensembled versions
    \[f_Q(z) = \Ep{Z\sim Q(\cdot\mid z)}{f(Z)}, \quad g_Q(z) = \Ep{Z\sim Q(\cdot\mid z)}{g(Z)}.\]
    Then
    \[\int_{\cZ} f(z)\cdot  [T_{P,Q}g](z)\;\mathsf{d}\mu(z) = \Ep{(Z_0,Z_1)\sim P}{\big(f_Q(Z_0)-f_Q(Z_1)\big)\cdot \big(g_Q(Z_0)-g_Q(Z_1)\big)}.\]
\end{lemma}
\noindent We also define an operator norm on $T_{P,Q}$ as
\[\|T_{P,Q}\|_{L_\infty\to L_1} = \sup\left\{\|T_{P,Q}f\|_{L_1} : \|f\|_{L_\infty} \leq 1\right\},\]
where for any measurable $f:\cZ\to\R$ we define $\|f\|_{L_1} = \int_{\cZ} |f(z)|\;\mathsf{d}\mu(z)$, and $\|f\|_{L_\infty} = \sup_{z \in \cZ} |f(z)|$.  By construction (and using the fact that $K_{P,Q}$ is positive semidefinite), we can equivalently write
\begin{equation}\label{eqn:T_norm_as_sup}
\|T_{P,Q}\|_{L_\infty\to L_1} = \sup\left\{ \int_{\cZ} \int_{\cZ} K_{P,Q}(z,z') f(z)f(z')\;\mathsf{d}\mu(z)\;\mathsf{d}\mu(z') : \|f\|_{L_\infty}\leq 1\right\}.
\end{equation}

\subsection{Stability guarantee}
\label{SubSec:StabGuarantee}

Our main result shows that the stability $\beta^2_P(\alg_Q)$ can be characterized by the norm of the associated operator $T_{P,Q}$.
\begin{theorem}\label{thm:main}
    For any $\alg:\cZ\to[0,1]$, it holds that
    \[\beta^2_P(\alg_Q) \leq \frac{1}{4}\|T_{P,Q}\|_{L_\infty\to L_1}.\]
\end{theorem}
In particular, note that this upper bound does not depend on the particular base algorithm $\alg$. As we will see below, this bound is tight when viewed as a result required to hold simultaneously over all $\alg$---in fact, it is possible to construct an algorithm $\alg$ and an ensembling channel $Q$ such that this upper bound is achieved for every~$P$. However, given a single algorithm $\alg$, the bound may be loose, i.e., a particular ensembled algorithm $\alg_Q$ may exhibit more stability than is guaranteed by this upper bound.

Theorem~\ref{thm:main} can also be extended to the setting where the output of our algorithm takes values in a separable Hilbert space; see Theorem~\ref{thm:main_hilbert} in Appendix~\ref{sec:HSExtension}.
\begin{proof}[Proof of Theorem~\ref{thm:main}]
Define
 $f(z) = \alg(z) - \frac{1}{2}$. 
Note that $\alg_Q(z) = f_Q(z)+\frac{1}{2} $ for all $z$, by construction. Then
\begin{align*}
    \beta^2_P(\alg_Q)
    &=\Ep{P}{\big(\alg_Q(Z_0)-\alg_Q(Z_1)\big)^2}\\
    &=\Ep{P}{\big(f_Q(Z_0)-f_Q(Z_1)\big)\cdot \big(f_Q(Z_0)-f_Q(Z_1)\big)}\\
    &= \int_{\cZ} f(z)\cdot  [T_{P,Q}f](z)\;\mathsf{d}\mu(z)\textnormal{ by Lemma~\ref{lem:operator_avg} (applied with $f=g$)}\\
    &\leq \|f\|_{L_\infty}\|T_{P,Q}f\|_{L_1}
    \leq \|f\|_{L_\infty}^2\|T_{P,Q}\|_{L_\infty\to L_1}\\
    &\leq \frac{1}{4}\|T_{P,Q}\|_{L_\infty\to L_1},
\end{align*}
where the last step holds since $f$ takes values in $[-\frac{1}{2},\frac{1}{2}]$ (because $\alg$ takes values in $[0,1]$).
\end{proof}

Next, we provide a useful relaxation of this bound, which allows us to work with a different operator norm. 
\begin{corollary}\label{cor:main_L2}
    Let $\psi$ be any density with respect to the base measure $\mu$, with $\psi(z)>0$ for all~$z$.
    Then
    for any $\alg:\cZ\to[0,1]$, it holds that
    \[\beta^2_P(\alg_Q) \leq \frac{1}{4}\|\tilde{T}^\psi_{P,Q}\|_{L_2\to L_2},\]
    where we define the operator $\tilde{T}^\psi_{P,Q}$ as
    \[[\tilde{T}^\psi_{P,Q}f](z) = \int_\cZ \frac{K_{P,Q}(z,z')}{\sqrt{\psi(z)\psi(z')}} \cdot f(z')\;\mathsf{d}\mu(z'). \]
\end{corollary}
\begin{proof}[Proof of Corollary~\ref{cor:main_L2}]
    Fix any $f:\cZ\to\R$ with $\|f\|_{L_\infty}\leq 1$.
    A straightforward calculation shows that
    \[\int_{\cZ} \int_{\cZ} K_{P,Q}(z,z') f(z)f(z')\;\mathsf{d}\mu(z)\;\mathsf{d}\mu(z') = \int_\cZ \tilde{f}(z)\cdot [\tilde{T}^\psi_{P,Q}\tilde{f}](z)\;\mathsf{d}\mu(z),\]
    where $\tilde{f}(z) = \sqrt{\psi(z)}\cdot f(z)$. Note that $\int_{\cZ} \tilde{f}(z)^2\;\mathsf{d}\mu(z) = \int_{\cZ} \psi(z) f(z)^2\;\mathsf{d}\mu(z)\leq 1$, since $\psi$ is a density and $\|f\|_{L_\infty}\leq 1$. 
Therefore,
    \[\int_{\cZ} \int_{\cZ} K_{P,Q}(z,z') f(z)f(z')\;\mathsf{d}\mu(z)\;\mathsf{d}\mu(z') \leq \|\tilde{T}^\psi_{P,Q}\|_{L_2\to L_2}.\]
    Since this holds for any $f$ with $\|f\|_{L_\infty}\leq 1$, applying~\eqref{eqn:T_norm_as_sup} we have
    \[\|T_{P,Q}\|_{L_\infty\to L_1}\leq \|\tilde{T}^\psi_{P,Q}\|_{L_2\to L_2}.\]
    Combined with Theorem~\ref{thm:main}, this completes the proof.
\end{proof}

\subsection{Universality}
Next, we show that ensembling is a \emph{universal} tool for obtaining stable algorithms. In particular, we will show that for any algorithm $\alg$, we can formulate $\alg$ as the ensembled version of some base algorithm $\alg^*$ such that our theory above exactly characterizes the stability of~$\alg$: that is, we construct an algorithm $\alg^*$ such that
\[\alg = [\alg^*]_Q,\]
and such that the stability $\beta^2_P(\alg)$ of the algorithm $\alg$ can be fully described by the result of Theorem~\ref{thm:main}. (We note however that the relaxation provided in Corollary~\ref{cor:main_L2} may not be tight.)

\begin{theorem}\label{thm:universality}
    Let $\alg:\cZ\to[0,1]$ be any algorithm. Then there exists another algorithm $\alg^*:\cZ\to[0,1]$ and an ensembling channel $Q(\cdot\mid z)$ such that
    \[\alg(z) = [\alg^*]_Q(z) \textnormal{ for all $z\in\cZ$}\]
    and such that
    \[\beta^2_P(\alg) = \frac{1}{4}\|T_{P,Q}\|_{L_\infty\to L_1}\textnormal{ for all distributions $P$ on $\cZ\times\cZ$}.\]
\end{theorem}

\section{Stability guarantees for examples}
\label{Sec:Examples}

Next we examine the implications of our main result, Theorem~\ref{thm:main}, in the settings of our examples. We will see how bagging (Example~\ref{exampleA}) can ensure stability with respect to deleting data (Example~\ref{example1}), and how smoothing by adding noise (Example~\ref{exampleB}) can ensure stability with respect to the corrupted data model (Example~\ref{example2}). We then also develop one additional example: stability with respect to missing data.

\subsection{Bagging \& deleted data (Example~\ref{example1}+\ref{exampleA})}\label{sec:example1+A}

We first need some notation: for any sequence of indices $r=(i_1,\dots,i_m)\in[n]^m$, we define the corresponding subvector of $z$ as
    \[z_r = (z_{i_1},\dots,z_{i_m}),\]
    and let $p_r$ denote the probability of selecting this `bag'. With this notation, the bagged algorithm is given by
\begin{equation}\label{eqn:define_A_bag}\alg_{\textnormal{bag}}(z) = \sum_r p_r \alg(z_r),\end{equation}
for a data vector $z=(z_1,\dots,z_n)$ of length $n$.  We can make this more concrete by considering the two most common examples:
\begin{itemize}
    \item For subbagging (sampling $m$ indices uniformly at random without replacement), we have $p_r = \frac{1}{n!/(n-m)!}$ for all $r\in[n]^m$ consisting of $m$ distinct indices, and $p_r = 0$ otherwise.
    \item For $m$ out of $n$ bootstrapping (sampling $m$ indices uniformly at random with replacement), we have $p_r = \frac{1}{n^m}$ for all $r\in[n]^m$.
\end{itemize}
We will assume that the bagged algorithm on the data vector $z_{-i}$ (which has length $n-1$) obeys
\begin{equation}\label{eqn:define_A_bag_i}\alg_{\textnormal{bag}}(z_{-i}) = \frac{\sum_r \one{i\not\in r}p_r \alg(z_r)}{\sum_r \one{i\not\in r}p_r},\end{equation}
that is, sampling a random bag from $[n]\setminus\{i\}$ is equivalent to sampling a bag $r$ with probabilities $p_r$, conditional on the event $i\not\in r$. This holds for both subbagging and bootstrapping, for any choice of $m$ (e.g., sampling $m$ entries from $z_{-i}$, without replacement, is equivalent to sampling $m$ entries without replacement from $z$ and conditioning on the event that entry $i$ was not sampled). 

The work of \citet{soloff2024bagging,soloff2024stability} shows a stability guarantee for $\alg_{\textnormal{bag}}$ for certain bagging schemes (we give details below).  We will now see that our general framework can generalize this result, while recovering the bound~\eqref{eqn:soloff_bound} as a special case.
To state our result, define $\pi_i = \sum_r \one{i\in r}p_r$ (the probability that a randomly sampled bag contains $i$), and $\Sigma_{ij} = \sum_r \one{i,j\in r} p_r - \pi_i\pi_j$ (the covariance between the events $i\in r$ and $j\in r$). 
\begin{proposition}\label{prop:example1}
    Let $\alg$ be any base algorithm returning outputs in $[0,1]$. The bagged algorithm $\alg_{\textnormal{bag}}$~\eqref{eqn:define_A_bag} satisfies
    \[\sup_{z=(z_1,\dots,z_n)}\frac{1}{n}\sum_{i=1}^n \big(\alg_{\textnormal{bag}}(z) - \alg_{\textnormal{bag}}(z_{-i})\big)^2 \leq \frac{1}{4}\|M_{\textnormal{bag}}\|_2,\]
    where $\|\cdot\|_2$ denotes the usual matrix operator norm, and where $M_{\textnormal{bag}}\in\R^{n\times n}$ is the matrix defined by entries
    \[(M_{\textnormal{bag}})_{ij} = \frac{1}{n} \frac{\Sigma_{ij}}{(1-\pi_i)(1-\pi_j)}.\]
\end{proposition}

We now see how \citet{soloff2024bagging,soloff2024stability}'s results can be recovered as a special case. Their work assumes a symmetry condition ensuring that the $n$ indices are treated interchangeably in the process of sampling a bag $r$, and moreover assumes correlation cannot be positive: 
\begin{equation}\label{eqn:soloff_assumptions}\pi_i = p \textnormal{ for all $i$}, \quad \Sigma_{ij} = -q \leq 0\textnormal{ for all $i\neq j$}\end{equation}
for some $p,q$ (this holds for both subbagging and bootstrapping). 
Under this condition, we can then calculate 
\[(M_{\textnormal{bag}})_{ij} = \begin{cases}\frac{1}{n} \cdot \frac{p}{1-p}, & i=j, \\ \frac{1}{n}\cdot \frac{-q}{(1-p)^2}, & i\neq j,\end{cases}\]
which 
implies
\[\|M_{\textnormal{bag}}\|_2 =  \frac{1}{n} \cdot \left( \frac{p}{1-p} + \frac{q}{(1-p)^2} \right) \leq \frac{1}{n-1}\cdot \frac{p}{1-p},\]
where the inequality holds since we must have $q\leq \frac{p(1-p)}{n-1}$ (because $\Sigma$ is positive semidefinite). Combined with Proposition~\ref{prop:example1}, this implies the stability bound
\begin{equation}\label{eqn:soloff_bound}\sup_{z=(z_1,\dots,z_n)}\frac{1}{n}\sum_{i=1}^n \big(\alg_{\textnormal{bag}}(z) - \alg_{\textnormal{bag}}(z_{-i})\big)^2 \leq  \frac{1}{4(n-1)} \cdot \frac{p}{1-p},\end{equation}
which exactly recovers the main result of \citet{soloff2024bagging,soloff2024stability}.

\subsection{Adding noise \& corrupted data (Example~\ref{example2}+\ref{exampleB})}
Next we derive stability bounds for the corrupted data setting.
For $x\in\R^s$, define
\[\Delta_h(x) = \textnormal{d}_{\chi^2}(x+W\| W)\textnormal{ where $W=(W_1,\dots,W_s)$ for $W_1,\dots,W_s\iidsim h$},\]
where we recall that $h$ is the density of the entries of the ensembling noise, and where $\textnormal{d}_{\chi^2}$ denotes the $\chi^2$-divergence. Essentially, $\Delta_h(x)$ is small if noise drawn from $h$ can mask a shift of magnitude $|x_i|$ across draws $i=1,\dots,s$.
\begin{proposition}\label{prop:example2}
In the setting and notation above, the smoothed algorithm $\alg_{\textnormal{smth}}$ satisfies 
    \[\frac{1}{{\binom{n}{s}}}\sum_{S\in\binom{[n]}{s}}\Ep{Z\sim \pi, \zeta\sim \nu}{\big(\alg_{\textnormal{smth}}(Z) - \alg_{\textnormal{smth}}(Z\circ\mathbf{1}_{S^c} + \zeta\circ\mathbf{1}_S)\big)^2}
    \leq \frac{s}{4n}\EE{\max_{S\in\binom{[n]}{s}}\Delta_h(\zeta_S-Z_S)}.\]
\end{proposition}
For example, suppose that $h$ is the density of the $\mathcal{N}(0,\sigma^2)$ distribution. Since
\[\textnormal{d}_{\chi^2}\big(\mathcal{N}(\mu,\sigma^2\mathbf{I}_s)\| \mathcal{N}(\mathbf{0}_s,\sigma^2\mathbf{I}_s)\big) = e^{\|\mu\|^2_2/\sigma^2}-1\]
for any $\mu\in\R^s$,
we therefore have 
\[\Delta_h(x) = e^{\|x\|^2_2/\sigma^2}-1,\]
and so the stability bound above can be simplified to
\[
\frac{s}{4n}\EE{ \exp\left\{\max_{S\in\binom{[n]}{s}} \|Z_S-\zeta_S\|^2_2/\sigma^2\right\}-1}.
\]
For instance, if $\|Z\|_\infty,\|\zeta\|_\infty\leq B$ almost surely (i.e., $\pi$ and $\nu$ are supported on $[-B,B]^n$), then
the stability result can be bounded by
\begin{equation}\label{eqn:example2_gaussian_bdd}
\frac{s\big(e^{4B^2s/\sigma^2}-1\big)}{4n}.
\end{equation}

As another example, we may consider adding more heavy-tailed noise, which will more effectively mask corruptions (and will allow us to avoid assuming boundedness). Let $h$ be the standard Cauchy distribution. By calculating the $\chi^2$ divergence between Cauchy distributions \citep{nielsen2022f}, we have
\[\Delta_h(x) = \frac{x^2}{2},\]
and
so the stability bound above can be simplified to
\[
\frac{s}{8n}\EE{\max_{S\in\binom{[n]}{s}} \|Z_S-\zeta_S\|^2_2}.
\]
For instance, if each $Z_i,\zeta_i$ is subgaussian, we then have a bound that is
\[\mathcal{O}\left(\frac{s^2\log n}{n}\right).\]

\subsection{Stability with respect to missing data}
We next develop an additional example, in the context of missing data. Let $Z$ denote the complete data (with no missingness), while $Z_\Omega$ denotes data with missingness, where $\Omega$ is the `mask' specifying which parts of the data are observed. Specifically, we assume $Z$ takes values in a space $\cZ= \R^{\cI}$, for instance:
\begin{itemize}
    \item For matrix-valued data $Z\in\R^{n\times m}$, we can equivalently write $Z\in\R^{[n]\times [m]}$, by representing a $n\times m$ matrix as a function mapping an index $(i,j)\in\cI=[n]\times[m]$ to a value $Z_{ij}\in \R$.
    \item $Z$ may consist of functional data: we may have $Z\in\R^{[0,1]}$ in the setting where we observe real-valued data over the time period $t\in[0,1]$ or $Z\in \R^{[0,1]\times[0,1]}$ if we observe real-valued data over a $2$-dimensional spatial domain.
\end{itemize}
The observation mask $\Omega$ is then a subset of $\cI$, specifying which part of $Z$ is observed (with $\Omega=\cI$ corresponding to the case of no missingness, i.e., $Z_{\cI}=Z$). Our goal is to ensure that $\alg(Z_\Omega)$, the output of the algorithm $\alg$ when the input $Z_\Omega$ has some missingness, is typically similar to $\alg(Z)$, the output when $\alg$ is trained on the fully observed data $Z$.

Now we formalize the general setting. Let $\pi$ be a joint distribution for the pair $(Z,\Omega)$. For convenience we assume that the possible values of $\Omega$, under this joint distribution $\pi$, is at most countably infinite, with $\Omega\in\{\omega_1,\omega_2,\dots\}$ (where $\omega_i\subseteq\cI$ for each $i\geq 1$). We will write~$\pi_Z$ to denote the marginal distribution of $Z$, and $\pi_i(z) = \Ppst{\pi}{\Omega=\omega_i}{Z=z}$ to denote the conditional distribution of $\Omega$, which in general depends on $Z$ (note that the common \emph{missing completely at random} assumption would instead require $\Omega\independent Z$).

Our goal is to be stable with respect to missingness, i.e., to bound
\[\Ep{(Z,\Omega)\sim\pi}{(\alg(Z)-\alg(Z_\Omega))^2}.\]
We now define an ensembling procedure that will allow for this type of stability guarantee. Let $q=(q_1,q_2,\dots)$ be a distribution on $\{\omega_1,\omega_2,\dots\}$. We ensemble via a \emph{Poissonized missingness} mechanism, by averaging over the following process: first sample $M\sim\textnormal{Poisson}(\lambda)$, then sample $\Omega_1,\dots,\Omega_M\iidsim q$ and use the observation mask $\cap_{j=1}^M\Omega_j$ (we take the convention that the empty intersection is the entire set $\cI$---that is, if $M=0$, then the observation mask is $\cI$, i.e., no missingness).
More formally, we can write
\begin{equation}\label{eqn:Poisson_missingness}\alg_{\textnormal{PM}(\lambda)}(z) = \Ep{M\sim\textnormal{Poisson}(\lambda),\Omega_j\iidsim q}{\alg(Z_{\cap_{j=1}^M \Omega_j})}.\end{equation}

This ensembling kernel leads to the following stability guarantee:
\begin{proposition}\label{prop:example3}
In the setting and notation above, then the algorithm $\alg_{\textnormal{PM}(\lambda)}$ constructed via Poissonized missingness satisfies 
    \[\Ep{(Z,\Omega)\sim \pi}{(\alg_{\textnormal{PM}(\lambda)}(z) - \alg_{\textnormal{PM}(\lambda)}(z_\Omega))^2}
    \leq \frac{\Ep{Z\sim \pi_Z}{\sup_i \{\pi_i(Z)/q_i\}}}{\lambda}.\]
\end{proposition}
\noindent In particular, if $\sup_i \{\pi_i(Z)/q_i\}$ is bounded (which may be possible if the dependence between~$\Omega$ and $Z$ is not too strong), we see that the stability guarantee holds with parameter $\asymp \lambda^{-1}$. 

For intuition, suppose that the original distribution $\pi$ typically leads to an $\epsilon$ portion of the data being missing (e.g., $\approx \epsilon n m$ entries are missing, from an $n\times m$ matrix). The Poissonized missingness mechanism will then lead to $\approx \lambda \epsilon$ missingness, since we have $\EE{M}=\lambda$ (e.g., $\approx \lambda \epsilon n m$ entries are missing, from an $n\times m$ matrix). This means that to ensure that stability holds with parameter $\asymp \lambda^{-1}$, we pay the price of increasing the fraction of missing data by the factor $\lambda$ in the ensembling procedure.

\section{Connections with privacy}
\label{Sec:Privacy}

Our framework for studying stability can be interpreted as saying that the channel $Q$ needs to be sufficiently noisy so as to `mask' the difference between the data $Z$ and its perturbed version $Z'$. On the surface, this intuition resembles the idea of privacy: it seems that we need the distributions $Q(\cdot\mid Z)$ and $Q(\cdot\mid Z')$ to be similar when $(Z,Z')\sim P$. Indeed, the following proposition makes this connection explicit:
\begin{proposition}\label{prop:compare_to_dtv} The operator $T_{P,Q}$ defined above satisfies
    \[\|T_{P,Q}\|_{L_\infty\to L_1} \leq 4\Ep{(Z,Z')\sim P}{\Big(\dtv(Q(\cdot \mid Z),Q(\cdot\mid Z'))\Big)^2},\]
    where $\dtv$ denotes the total variation distance.
\end{proposition}
This upper bound relates to the statistical framework of \emph{total variation privacy} \citep{barber2014privacy}: for instance, for the setting of deleting one data point, we might impose a privacy constraint
\begin{equation}\label{eqn:tv_privacy}\dtv(Q(\cdot\mid z),Q(\cdot\mid z_{-i}))\leq \epsilon \textnormal{ for all $z=(z_1,\dots,z_n)$ and all $i\in[n]$}.\end{equation}
(This is strictly weaker than the more commonly used \emph{differential privacy} condition \citep{dwork2006calibrating}; see \citet{dwork2014algorithmic} and \citet{su2025statistical} for an overview of various privacy frameworks.)

Under a total variation privacy assumption~\eqref{eqn:tv_privacy}, Proposition~\ref{prop:compare_to_dtv} ensures $\|T_{P,Q}\|_{L_\infty\to L_1} \leq 4\epsilon^2$, which then leads to a stability guarantee $\beta^2_P(\alg_Q)\leq \epsilon^2$ by Theorem~\ref{thm:main}.
However, the inequality in Proposition~\ref{prop:compare_to_dtv} is often extremely loose. 
We can see this in the settings of our two examples:
\begin{customexample}{\ref{example1}+\ref{exampleA}}
Recall that $Q(\cdot\mid z)$ is the distribution over all bags $r$ given by probabilities $p_r$ (as in~\eqref{eqn:define_A_bag}), while $Q(\cdot\mid z_{-i})$ is the same distribution conditioned on the event $i\not\in r$ (as in~\eqref{eqn:define_A_bag_i}). Since this event has probability $1-\pi_i$, we can therefore calculate
$\dtv(Q(\cdot\mid z),Q(\cdot\mid z_{-i})) = \pi_i$.
In particular, in the special case of assumptions~\eqref{eqn:soloff_assumptions},
\[\dtv(Q(\cdot\mid z),Q(\cdot\mid z_{-i})) = p\]
for all $i\in[n]$. 
Therefore, Proposition~\ref{prop:compare_to_dtv} yields the bound
\[\sup_{z=(z_1,\dots,z_n)}\frac{1}{n}\sum_{i=1}^n \big(\alg(z) - \alg(z_{-i})\big)^2 \leq p^2.\]
This is a much weaker bound than the stability guarantee~\eqref{eqn:soloff_bound} (which recovers the upper bound $ \frac{1}{4(n-1)} \cdot \frac{p}{1-p}$ established by \citet{soloff2024bagging}). In particular, this bound does not decrease with $n$, and thus illustrates the looseness of Proposition~\ref{prop:compare_to_dtv}.
\end{customexample}

\begin{customexample}{\ref{example2}+\ref{exampleB}}
Recall that $Q(w\mid Z) = \prod_{i=1}^n h(w_i-Z_i)$. Let $Z' = Z\circ\mathbf{1}_{S^c} + \zeta\circ\mathbf{1}_S$ denote the corrupted version of the data $Z$.
Then we can calculate
\[\dtv\big(Q(\cdot\mid Z),Q(\cdot\mid Z')\big) = \Delta^{\textnormal{TV}}_h(\zeta_S-Z_S),\]
where we define
\[\Delta^{\textnormal{TV}}_h(x) = \dtv(x+W\| W)\textnormal{ with $W=(W_1,\dots,W_s)$ for $W_1,\dots,W_s\iidsim h$}\]
for $x\in\R^s$.
For instance, if $h$ is given by the $\mathcal{N}(0,\sigma^2)$ density and we assume $\|Z\|_\infty,\|\zeta\|_\infty\leq B$ (as in the discussion after Proposition~\ref{prop:example2}), we then have $\Delta^{\textnormal{TV}}_h(x)\leq \|x\|_2/\sqrt{2\pi}\sigma$ and so $\dtv\big(Q(\cdot\mid Z),Q(\cdot\mid Z')\big)\leq 2B\sqrt{s}/\sqrt{2\pi}\sigma$.  Therefore, Proposition~\ref{prop:compare_to_dtv} implies
\[\frac{1}{n}\sum_{i=1}^n\Ep{Z\sim \pi, \zeta\sim \nu}{\big(\alg_{\textnormal{smth}}(Z) - \alg_{\textnormal{smth}}(Z\circ\mathbf{1}_{S^c}+\zeta\circ\mathbf{1}_S\big)^2} \leq \frac{8B^2s}{\pi \sigma^2}.\]
Since this bound does not decrease with $n$, this is a weaker bound than the stability guarantee obtained in~\eqref{eqn:example2_gaussian_bdd} under the same assumptions, thus providing another illustration of the gap in Proposition~\ref{prop:compare_to_dtv}.
\end{customexample}

In other words, we can see via these examples that it is not the case that stability is simply a consequence of (total variation) privacy.  The stability achieved by ensembling (which is characterized by $\|T_{P,Q}\|_{L_\infty\to L_1}$, as in the results of Theorem~\ref{thm:main}) may often be much stronger than the amount of stability implied via privacy (i.e., via the bound in Proposition~\ref{prop:compare_to_dtv}).

\section{Discussion}
\label{Sec:Discussion}

Our goal in this work has been to introduce a unified framework that quantifies the stability guarantees afforded by ensembling.  We have seen that different types of data perturbation can be stabilized through different ensembling strategies.

These results naturally lead to some important open questions. First, our theoretical guarantees rely on a boundedness assumption: we assume in Theorem~\ref{thm:main} that $\alg$ returns outputs lying in $[0,1]$ (or, in the extension to Hilbert space-valued output in Appendix~\ref{sec:HSExtension}, that the output lies in a bounded subset of a Hilbert space). While this is reasonable for certain applications, relaxing these conditions or allowing for data-dependent bounds is an important direction for future work (see \citet{soloff2024bagging} for some extensions of this type in the specific setting of bagging for stability with respect to data deletion).

Next, the goal of ensembling in this setting is to provide stability while retaining the good performance of the original algorithm $\alg$---that is, we would like to return output that is similar to the output of the base algorithm $\alg$, if possible (without this aim, one could achieve stability trivially: simply return a constant output, or the output of a provably stable algorithm, while ignoring $\alg$). It would be interesting to quantify the extent to which different ensembling strategies navigate this tradeoff between the resulting stability guarantee, and the deviation from the original $\alg$. 

Finally, there are many practical notions of stability that we have not considered as examples here: for instance, stability with respect to a choice of prior; with respect to a choice of bandwidth or tuning parameter; with respect to the initialization point in an optimization problem; etc. In our framework, the `data' $Z$ can denote any input to the algorithm---that is, $Z$ can contain a choice of tuning parameters, etc, in addition to the random variables that comprise the input data. Therefore, it would be interesting to study whether our framework can give meaningful stability guarantees for questions of this type.

\subsection*{Acknowledgements}
R.F.B. was supported by the Office of Naval Research via grant N00014-24-1-2544. R.J.S. was supported by European Research Council Advanced Grant 101019498.
\bibliographystyle{plainnat}
\bibliography{bib}

\appendix

\section{Proofs}\label{sec:proofs_appendix}
\subsection{Proofs of main results}

\begin{proof}[Proof of Lemma~\ref{lem:operator_avg}] We calculate
    \begin{align*}
    &\int_{\cZ} f(z)\cdot  [T_{P,Q}g](z)\;\mathsf{d}\mu(z)\\
    &=\int_{\cZ} f(z) \left[\int_{\cZ} K_{P,Q}(z,z')\cdot g(z')\;\mathsf{d}\mu(z')\right]\;\mathsf{d}\mu(z)\\
    &=\int_{\cZ} \int_{\cZ} \Ep{P}{\big(Q(z\mid Z_0) - Q(z\mid Z_1)\big)\cdot \big(Q(z'\mid Z_0) - Q(z'\mid Z_1)\big)}\cdot f(z)g(z')\;\mathsf{d}\mu(z')\;\mathsf{d}\mu(z)\\
    &=\Ep{P}{\int_{\cZ} \int_{\cZ} \big(Q(z\mid Z_0) - Q(z\mid Z_1)\big)\cdot \big(Q(z'\mid Z_0) - Q(z'\mid Z_1)\big)\cdot f(z)g(z')\;\mathsf{d}\mu(z')\;\mathsf{d}\mu(z)}\\
    &=\Ep{P}{\int_{\cZ} \big(Q(z\mid Z_0) - Q(z\mid Z_1)\big)\cdot f(z)\;\mathsf{d}\mu(z) \cdot \int_{\cZ} \big(Q(z'\mid Z_0) - Q(z'\mid Z_1)\big)\cdot g(z')\;\mathsf{d}\mu(z')}\\
    &=\Ep{P}{\big(f_Q(Z_0)-f_Q(Z_1)\big)\cdot \big(g_Q(Z_0)-g_Q(Z_1)\big)},
\end{align*}
where the second step holds by definition of $K_{P,Q}$, the third step holds by Fubini's theorem, and the last step holds by definition of $f_Q,g_Q$.
\end{proof}

\begin{proof}[Proof of Theorem~\ref{thm:universality}]

We can assume $|\cZ|>1$ to avoid the trivial case.
Let $z^{(0)},z^{(1)}\in\cZ$ be distinct points. Define
\[\alg^*(z) = \One{z = z^{(1)}},\]
and define an ensembling channel
\[Q(\cdot\mid z) = (1-\alg(z)) \cdot \delta_{z^{(0)}} + \alg(z) \cdot \delta_{z^{(1)}}.\]
For any input data $z$, this channel places all its mass on the two data points, $z^{(0)}$ and $z^{(1)}$, with weights determined by the output $\alg(z)$ of the original algorithm. (Recall that we require that $Q(\cdot\mid z)$ is a density with respect to a base measure $\mu$, for all $z$; this holds by choosing $\mu$ to be the counting measure.) 
We then calculate that, for any $z\in\cZ$,
\[[\alg^*]_Q(z) = \Ep{Z\sim Q(\cdot\mid z)}{\alg^*(Z)} = \Ep{Z\sim Q(\cdot\mid z)}{\One{Z=z^{(1)}}} = \Pp{Z\sim Q(\cdot\mid z)}{Z=z^{(1)}} = \alg(z).\]
In other words, $\alg = [\alg^*]_Q$, i.e., the original algorithm $\alg$ is equivalent to the ensembled version of the base algorithm $\alg^*$.

Next, let $P$ be any distribution on $\cZ$.
By Theorem~\ref{thm:main} we know that
\[\beta^2_P(\alg) = \beta^2_P([\alg^*]_Q) \leq \frac{1}{4}\|T_{P,Q}\|_{L_\infty\to L_1}.\]
We next need to show that this is in fact an equality. We first calculate the associated kernel $K_{P,Q}$. 
We have
\begin{multline*}K_{P,Q}(z^{(1)},z^{(1)}) = \Ep{(Z_0,Z_1)\sim P}{\big(Q(z^{(1)}\mid Z_0) - Q(z^{(1)}\mid Z_1)\big)^2}\\ = \Ep{(Z_0,Z_1)\sim P}{\big(\alg(Z_0)-\alg(Z_1)\big)^2} = \beta^2_P(\alg),\end{multline*}
by definition of $Q$.
Similar calculations verify that
\[K_{P,Q}(z^{(0)},z^{(0)}) =  \beta^2_P(\alg), \quad 
K_{P,Q}(z^{(0)},z^{(1)}) = K_{P,Q}(z^{(1)},z^{(0)}) =  -\beta^2_P(\alg). \]
And, $K_{P,Q}(z,z')=0$ for any other pair $(z,z')$, since $Q$ only returns outputs in $\{z^{(0)},z^{(1)}\}$.

Next we calculate the operator $T_{P,Q}$.
For any function $f:\cZ\to\R$, by definition of $T_{P,Q}$ we can verify that
\[[T_{P,Q}f](z) = \begin{cases}\beta^2_P(\alg) \cdot (f(z^{(0)})-f(z^{(1)})), & z=z^{(0)},\\-\beta^2_P(\alg) \cdot (f(z^{(0)})-f(z^{(1)})) , & z=z^{(1)}, \\ 0, & \textnormal{ otherwise}.\end{cases}\]
Therefore, for any $f$ with $\|f\|_{L_\infty}\leq 1$,
\[\|T_{P,Q}f\|_{L_1} = 2\beta^2_P(\alg) \cdot |f(z^{(0)})-f(z^{(1)})| \leq 4\beta^2_P(\alg) \cdot\|f\|_{L_\infty}\leq 4\beta^2_P(\alg),\]
which completes the proof.
\end{proof}

\subsection{Proofs for examples}

\begin{proof}[Proof of Proposition~\ref{prop:example1}]
First we recall how this example can be written in our unified notation. Given the definition of the bagged algorithm $\alg_{\textnormal{bag}}$, we can see that the ensembling channel $Q$ satisfies
\[Q(\cdot\mid z) = \sum_r p_r \delta_{z_r}, \quad Q(\cdot\mid z_{-i}) = \frac{\sum_r \one{i\not\in r}p_r \delta_{z_r}}{\sum_r  \one{i\not\in r}p_r}.\]
(We can view these as densities with respect to the counting measure.)
We also have 
\[P = \frac{1}{n}\sum_{i=1}^n\delta_{(z,z_{-i})},\]
since the perturbed data is determined by deleting one index $i$ at random. 

Below, we will verify that the resulting operator $T_{P,Q}$ can be characterized by the following identity: for any bounded measurable functions $f,g$,
\begin{equation}\label{eqn:convert_example1_to_matrix}
    \int_{z'} f(z')\cdot  [T_{P,Q}g](z')\;\mathsf{d}\mu(z') = \sum_{i=1}^n\left(\sum_r M_{i,r} f(z_r)\right)\cdot \left( \sum_r M_{i,r} g(z_r)\right),
\end{equation}
where we define a matrix $M$ with entries\footnote{To define this matrix $M$ formally, note that the bag $r$ takes values in the space $\textnormal{seq}_{[n]}$, denoting the set of all sequences $(i_1,\dots,i_k)$ with entries in $[n]$, of arbitrary length $k$. Therefore $M$ is an array indexed by $(i,r)\in[n]\times \textnormal{seq}_{[n]}$, and in particular, has countably infinitely many columns. In practice, most common bagging schemes do not allow for bags of unbounded length and thus $M$ can be viewed as finite-dimensional (i.e., only finitely many columns of $M$ can contain any nonzero values), but the calculations here hold regardless of whether the dimension of $M$ is finite or countably infinite.}
\[M_{i,r} =  \frac{1}{\sqrt{n}} \cdot p_r \left( 1 - \frac{\one{i\not \in r} }{\sum_{r'} \one{i\not \in r'} p_{r'} }\right).\]
Next let $v,w$ be vectors with entries $v_r = f(z_r)$, $w_r = g(z_r)$. Therefore,
\[\int_{z'} f(z')\cdot  [T_{P,Q}g](z')\;\mathsf{d}\mu(z') = \sum_{i=1}^n\left(\sum_r M_{i,r} v_r\right)\cdot \left( \sum_r M_{i,r} w_r\right) = v^\top M^\top M w .\]
Therefore,
\[\sup_{\|f\|_{L_\infty},\|g\|_{L_\infty}\leq 1} \int_{z'} f(z')\cdot  [T_{P,Q}g](z')\;\mathsf{d}\mu(z') = \sup_{\|v\|_\infty,\|w\|_\infty\leq 1}v^\top M^\top M w,\]
or in other words,
\begin{equation}\label{eqn:matrix_exact}\|T_{P,Q}\|_{L_\infty\to L_1} = \|M^\top M\|_{\infty\to 1},\end{equation}
where $\|A\|_{\infty\to 1} = \sup_{\|w\|_\infty\leq 1}\|A w\|_1$.

Next, define another matrix $\tilde{M}$ with entries
\[\tilde{M}_{i,r} =  \frac{1}{\sqrt{n}} \cdot \sqrt{p_r} \left( 1 - \frac{\one{i\not \in r}}{\sum_{r'} \one{i\not \in r'} p_{r'} }\right),\]
so that $M_{i,r} = \sqrt{p_r}\tilde{M}_{i,r}$. Since $(p_r)$ is a probability vector, a straightforward calculation shows that\footnote{We remark that this step, where we bound an $L_\infty\to L_1$ norm via an $L_2\to L_2$ norm, is analogous to the relaxation of Theorem~\ref{thm:main} given in Corollary~\ref{cor:main_L2}; here we choose to show these steps directly, rather than relying on the corollary, since the expression~\eqref{eqn:matrix_exact} gives an interpretable exact calculation of the stability properties for this example.}
\[\|M^\top M\|_{\infty\to 1}
\leq \|\tilde{M}^\top \tilde{M}\|_2 = \|\tilde{M}\tilde{M}^\top\|_2.\] 
We can also calculate, for each $i,j\in[n]$,
\[(\tilde{M}\tilde{M}^\top)_{ij} = \frac{1}{n}\sum_r p_r \left( 1 - \frac{\one{i\not \in r} }{\sum_{r'} \one{i\not \in r'} p_{r'} }\right)\cdot \left( 1 - \frac{\one{j\not \in r} }{\sum_{r'} \one{j\not \in r'} p_{r'} }\right) = (M_{\textnormal{bag}})_{ij},\]
and so combining everything, we have showed that
\[\|T_{P,Q}\|_{L_\infty\to L_1}\leq\|M_{\textnormal{bag}}\|_2.\]
Applying Theorem~\ref{thm:main}, this verifies the stability claim for $\alg_{\textnormal{bag}}$.

To complete the proof, we need to verify the identity stated in~\eqref{eqn:convert_example1_to_matrix}. First, let $f_Q$ be defined as in Lemma~\ref{lem:operator_avg}: that is, we have
\[f_Q(z') = \Ep{Z\sim Q(\cdot\mid z')}{f(Z)}.\]
We can evaluate this function at $z'=z$ as
\[f_Q(z) = \Ep{Z\sim Q(\cdot\mid z)}{f(Z)} = \sum_r Q(z_r\mid z) \cdot f(z_r) = \sum_r p_r f(z_r),\]
and similarly at $z'=z_{-i}$ (for any $i\in[n]$) as
\[f_Q(z) = \Ep{Z\sim Q(\cdot\mid z)}{f(Z)} = \sum_r Q(z_r\mid z_{-i}) \cdot f(z_r) = \frac{\sum_r\one{i\not\in r} p_r f(z_r)}{\sum_r \one{i\not\in r} p_r},\]
by definition of $Q(\cdot\mid z)$ and $Q(\cdot\mid z_{-i})$. The analogous calculations hold for $g_Q$. Then
\begin{align*}
    &\int_{z'} f(z')\cdot  [T_{P,Q}g](z')\;\mathsf{d}\mu(z') \\
    &= \Ep{(Z_0,Z_1)\sim P}{\big(f_Q(Z_0)-f_Q(Z_1)\big)\cdot \big(g_Q(Z_0)-g_Q(Z_1)\big)}\textnormal{ by Lemma~\ref{lem:operator_avg}}\\
    &= \frac{1}{n}\sum_{i=1}^n\big(f_Q(z)-f_Q(z_{-i})\big)\cdot \big(g_Q(z)-g_Q(z_{-i})\big)\\
    &= \frac{1}{n}\sum_{i=1}^n\left(\sum_r p_r f(z_r) - \frac{\sum_r \one{i\not \in r} p_r f(z_r)}{\sum_r \one{i\not \in r} p_r }\right)\cdot \left( \sum_r p_r g(z_r) - \frac{\sum_r \one{i\not \in r} p_r g(z_r)}{\sum_r \one{i\not \in r} p_r }\right)\\
    &= \frac{1}{n}\sum_{i=1}^n\left(\sum_r \left(1 - \frac{\one{i\not\in r}}{\sum_{r'}\one{i\not\in r'} p_{r'}}\right) \cdot p_r f(z_r) \right)\cdot\left(\sum_r \left(1 - \frac{\one{i\not\in r}}{\sum_{r'}\one{i\not\in r'} p_{r'}}\right) \cdot p_r g(z_r) \right)\\
    &= \sum_{i=1}^n\left(\sum_r M_{i,r} f(z_r)\right)\cdot \left( \sum_r M_{i,r} g(z_r)\right),
\end{align*}
by definition of $M$.
\end{proof}

\begin{proof}[Proof of Proposition~\ref{prop:example2}]

First, recall that by~\eqref{eqn:T_norm_as_sup}, it holds that
\[\|T_{P,Q}\|_{L_\infty\to L_1} = \sup\left\{ \int_{\cZ} \int_{\cZ} f(z_0)f(z_1)K_{P,Q}(z_0,z_1)\;\mathsf{d}\mu(z_0)\;\mathsf{d}\mu(z_1)\right\},\]
where the supremum is taken over all functions $f$ with $\|f\|_{L_\infty} \leq 1$. We will now fix any such function $f$, and will bound this integral for the example.

Let $z_0,z_1\in\R^n$.  Recall that when $(Z,Z') \sim P$, the perturbed data point $Z'$ is defined by replacing $Z_i$ with $\zeta_i$ for all $i\in S$, where $S \in \binom{[n]}{s}$ is sampled uniformly at random, and where $Z\sim \pi,\zeta\sim\nu$. Then
\[\frac{Q(z_0\mid Z')}{Q(z_0\mid Z)} = \frac{\prod_{i=1}^n h(z_{0i} - Z'_i)}{\prod_{i=1}^n h(z_{0i}-Z_i)} = \frac{\prod_{i\in S} h(z_{0i} - \zeta_i)}{\prod_{i\in S} h(z_{0i}-Z_i)},\]
and similarly for $z_1$ in place of $z_0$, which yields
\begin{align*}
    &K_{P,Q}(z_0,z_1)
    =\Ep{(Z,Z')\sim P}{(Q(z_0\mid Z)-Q(z_0\mid Z'))(Q(z_1\mid Z) - Q(z_1\mid Z'))}\\
    &=\Ep{(Z,Z')\sim P}{Q(z_0\mid Z)Q(z_1\mid Z)\left(\frac{\prod_{i=1}^n h(z_{0i} - Z'_i)}{\prod_{i=1}^n h(z_{0i}-Z_i)}-1\right)\left(\frac{\prod_{i=1}^n h(z_{1i} - Z'_i)}{\prod_{i=1}^n h(z_{1i}-Z_i)}-1\right)}\\
    &=\frac{1}{\binom{n}{s}}\sum_{S\in\binom{[n]}{s}}\Ep{\substack{Z\sim \pi\\\zeta\sim \nu}}{Q(z_0\mid Z)Q(z_1\mid Z)\left(\frac{\prod_{i\in S} h(z_{0i} - \zeta_i)}{\prod_{i\in S} h(z_{0i}-Z_i)}-1\right)\left(\frac{\prod_{i\in S} h(z_{1i} - \zeta_i)}{\prod_{i\in S} h(z_{1i}-Z_i)}-1\right)}.
\end{align*}
Therefore, applying Fubini's theorem, we can write
\begin{multline*}
    \int_{\cZ} \int_{\cZ} f(z_0)f(z_1)K_{P,Q}(z_0,z_1)\;\mathsf{d}\mu(z_0)\;\mathsf{d}\mu(z_1)\\
    = \frac{1}{\binom{n}{s}}\sum_{S\in\binom{[n]}{s}}\EE{f(Z_0)f(Z_1) \cdot \left(\frac{\prod_{i\in S} h(Z_{0i} - \zeta_i)}{\prod_{i\in S} h(Z_{0i}-Z_i)}-1\right)\left(\frac{\prod_{i\in S} h(Z_{1i} - \zeta_i)}{\prod_{i\in S} h(Z_{1i}-Z_i)}-1\right)},
\end{multline*}
where the expected value is taken with respect to 
\[Z\sim \pi, \quad \zeta\sim\nu, \quad Z_0,Z_1 \mid Z \iidsim Q(\cdot\mid Z).\]
Since $Z_0,Z_1$ are conditionally i.i.d.\ given $Z,\zeta$, we can rewrite this again as
\[ 
\frac{1}{\binom{n}{s}}\sum_{S\in\binom{[n]}{s}}\EE{\EEst{f(Z_0) \cdot \left(\frac{\prod_{i\in S} h(Z_{0i} - \zeta_i)}{\prod_{i\in S} h(Z_{0i}-Z_i)}-1\right)}{Z,\zeta}^2}.
\]
Next, define
\[f_S(Z_{0S}) = \EEst{f(Z_0)}{Z_{0S},Z,\zeta}.\]
By the tower law, for each $S$,
\begin{align*}
    &\EEst{f(Z_0) \cdot \left(\frac{\prod_{i\in S} h(Z_{0i} - \zeta_i)}{\prod_{i\in S} h(Z_{0i}-Z_i)}-1\right)}{Z,\zeta}\\
    &=\EEst{\EEst{f(Z_0) \cdot \left(\frac{\prod_{i\in S} h(Z_{0i} - \zeta_i)}{\prod_{i\in S} h(Z_{0i}-Z_i)}-1\right)}{Z_{0S},Z,\zeta}}{Z,\zeta}\\
    &=\EEst{\EEst{f(Z_0)}{Z_{0S},Z,\zeta} \cdot \left(\frac{\prod_{i\in S} h(Z_{0i} - \zeta_i)}{\prod_{i\in S} h(Z_{0i}-Z_i)}-1\right)}{Z,\zeta}\\
    &=\EEst{f_S(Z_{0S})\cdot \left(\frac{\prod_{i\in S} h(Z_{0i} - \zeta_i)}{\prod_{i\in S} h(Z_{0i}-Z_i)}-1\right)}{Z,\zeta}.
\end{align*}
Next consider the random variable $\frac{\prod_{i\in S} h(Z_{0i} - \zeta_i)}{\prod_{i\in S} h(Z_{0i}-Z_i)}-1$. By construction, we have
\[\EEst{\frac{\prod_{i\in S} h(Z_{0i} - \zeta_i)}{\prod_{i\in S} h(Z_{0i}-Z_i)}-1}{Z,\zeta}=0, \quad \EEst{\left(\frac{\prod_{i\in S} h(Z_{0i} - \zeta_i)}{\prod_{i\in S} h(Z_{0i}-Z_i)}-1\right)^2}{Z,\zeta} = \Delta_h(Z_S-\zeta_S).\]
Therefore, by Cauchy--Schwarz,
\[\EEst{f_S(Z_{0S})\cdot \left(\frac{\prod_{i\in S} h(Z_{0i} - \zeta_i)}{\prod_{i\in S} h(Z_{0i}-Z_i)}-1\right)}{Z,\zeta}\leq \textnormal{Var}(f_S(Z_{0S})\mid Z,\zeta)^{1/2} \cdot \Delta_h(Z_S-\zeta_S)^{1/2}.\]
Returning to our work above we then have
\begin{multline*}
    \int_{\cZ} \int_{\cZ} f(z_0)f(z_1)K_{P,Q}(z_0,z_1)\;\mathsf{d}\mu(z_0)\;\mathsf{d}\mu(z_1)\\
    \leq  \frac{1}{\binom{n}{s}}\sum_{S\in\binom{[n]}{s}}\EE{\textnormal{Var}(f_S(Z_{0S})\mid Z,\zeta) \cdot \Delta_h(Z_S-\zeta_S)}\\
    \leq  \EE{\frac{1}{\binom{n}{s}}\sum_{S\in\binom{[n]}{s}}\textnormal{Var}(f_S(Z_{0S})\mid Z,\zeta) \cdot \max_{S\in\binom{[n]}{s}}\Delta_h(Z_S-\zeta_S)}.
\end{multline*}
Since we have  $f_S(Z_{0S}) = \EEst{f(Z_0)}{Z_{0S},Z,\zeta}$ by definition, where $Z_{01},\dots,Z_{0n}$ are independent conditional on $Z,\zeta$, we can apply Lemma~\ref{lem:sum_var_subset} (given below), conditional on $Z,\zeta$ to obtain
\[
\frac{1}{\binom{n}{s}}\sum_{S\in\binom{[n]}{s}}\textnormal{Var}(f_S(Z_{0S})\mid Z,\zeta) \leq \frac{s}{n} \textnormal{Var}(f(Z_0)\mid Z,\zeta).
\]
And, since $\|f\|_{L_\infty}\leq 1$, $\textnormal{Var}(f(Z_0)\mid Z,\zeta)\leq 1$. Therefore,
\[\int_{\cZ} \int_{\cZ} f(z_0)f(z_1)K_{P,Q}(z_0,z_1)\;\mathsf{d}\mu(z_0)\;\mathsf{d}\mu(z_1)\leq \frac{s}{n}\EE{\max_{S\in\binom{[n]}{s}}\Delta_h(Z_S-\zeta_S)}.\]
In combination with Theorem~\ref{thm:main}, this completes the proof.
\end{proof}

\begin{proof}[Proof of Proposition~\ref{prop:example3}]
First we relate this example to our unified notation. The distribution $P$ on $(Z,Z')$ can be defined as follows:
\[\textnormal{$P$ is the distribution of $(Z,Z_\Omega)$ induced by $(Z,\Omega)\sim\pi$}.\]
The ensembling kernel $Q$ is given by
\[Q(\cdot\mid z) = \sum_{m\geq 0} \sum_{j_1,\dots,j_m\geq 1} \frac{\lambda^me^{-\lambda}}{m!} \cdot \prod_{k=1}^m q_{j_k}\cdot  \delta_{z_{\cap_{k=1}^m \omega_{j_k}}}.\]
(Note that for any $z$, $Q(\cdot \mid z)$ can be expressed as a density with respect to the counting measure $\mu$ on $\cZ$, since it is a discrete distribution.)

Now we rewrite the ensembling procedure one more time. Given a binary vector $\bfa=(a_1,a_2,\dots)$ with finitely many $1$s, let $\omega_{\bfa}$ denote the intersection of all selected masks as indicated by the $1$s, i.e.,
\[\omega_{\bfa} = \cap_{j\geq 1, a_j = 1} \omega_j.\]
(As before, we allow for an empty intersection: if $\bfa=(0,0,\dots)$ then $\omega_{\bfa}=\cI$, i.e., no missingness.) 

Under the Poissonized missingness mechanism, to generate a sample from $Q(\cdot\mid z)$, we return $z_\Omega$ where, to construct $\Omega$, we draw $M\sim\textnormal{Poisson}(\lambda)$ and then sample $M$ times from the $\textnormal{Multinomial}(q)$ distribution. By properties of the Poisson and Multinomial, this is equivalent to sampling $n_j\sim\textnormal{Poisson}(\lambda q_j)$, independently for each $j\geq 1$, and then the set $\{j_1,\dots,j_M\}$ consists of (a random permutation of) these selected indices (i.e., each index $j$ is included $n_j$ times). By construction, the output $z_\Omega$ of the ensembling kernel $Q$ is therefore determined by
\[\Omega = \cap_{k=1}^M \omega_{j_k} = \omega_{\bfa}\textnormal{ where }\bfa = (\one{n_1\geq1},\one{n_2\geq1},\dots).\]
In particular, we have $a_j\sim\textnormal{Bernoulli}(1-e^{-\lambda q_j})$, independently for each $j\geq 1$, since $\PP{n_j\geq 1} = \PP{\textnormal{Poisson}(\lambda q_j)\geq 1} = 1-e^{-\lambda q_j}$. Therefore, we can write
\[Q(\cdot\mid z) = \sum_{\bfa} \prod_{j\geq 1, a_j=1}(1-e^{-\lambda q_j}) \cdot \prod_{j\geq 1, a_j = 0}e^{-\lambda q_j} \cdot \delta_{z_{\omega_{\bfa}}} = \sum_{\bfa}  e^{-\lambda} \cdot \prod_{j\geq 1} (e^{\lambda q_j}-1)^{a_j} \cdot \delta_{z_{\omega_{\bfa}}}.\]
In particular, given $(Z,\Omega)\sim\pi$, we can calculate
\[Q(\cdot\mid Z) = \sum_{\bfa} e^{-\lambda}\prod_{j\geq 1}(e^{\lambda q_j}-1)^{a_j} \cdot \delta_{Z_{\omega_{\bfa}}},\]
and, on the event $\Omega=\omega_i$ (for each $i\geq 1$),
\begin{multline*}Q(\cdot\mid Z_\Omega)=Q(\cdot\mid Z_{\omega_i}) = \sum_{\bfa} e^{-\lambda}\prod_{j\geq 1}(e^{\lambda q_j}-1)^{a_j} \cdot \delta_{Z_{\omega_{\bfa}\cap \omega_i}}  \\= \sum_{\bfa} \frac{\one{a_i=1}}{1-e^{-\lambda q_i}}\cdot  e^{-\lambda}\prod_{j\geq 1}(e^{\lambda q_j}-1)^{a_j} \cdot \delta_{Z_{\omega_{\bfa}}}.\end{multline*}

From this point on, the proof follows a similar structure as the proof of Proposition~\ref{prop:example1}.
    In this example, the operator $T_{P,Q}$ can be characterized by the following identity: for any bounded measurable functions $f,g$,
\begin{equation}\label{eqn:convert_example3_to_matrix}
    \int_{\mathcal{Z}} f(z')\cdot  [T_{P,Q}g](z')\;\mathsf{d}\mu(z') = \Ep{Z\sim\pi_Z}{\sum_{i\geq 1}\left(\sum_{\bfa} M_{i,\bfa}(Z) f(Z_{\omega_{\bfa}})\right)\cdot \left( \sum_{\bfa} M_{i,\bfa}(Z) g(Z_{\omega_{\bfa}})\right)},
\end{equation}
where for each $z\in\cZ$, we define a (countably-infinite-dimensional) matrix $M(z)$ with entries
\[M_{i,\bfa}(z) = \sqrt{\pi_i(z)} \cdot e^{-\lambda} \cdot \prod_{j\geq 1} (e^{\lambda q_j}-1)^{a_j} \cdot  \left( 1 - \frac{\one{a_i=1}}{1-e^{-\lambda q_i}}\right),\]
indexed by $i\geq 1$ and by infinite binary vectors $\bfa$ with finitely many $1$s. (The proof of the claim~\eqref{eqn:convert_example3_to_matrix} is analogous to the proof of~\eqref{eqn:convert_example1_to_matrix} in Proposition~\ref{prop:example1}.) Similarly to the proof of Proposition~\ref{prop:example1}, this establishes that
\[\|T_{P,Q}\|_{L_\infty\to L_1} \leq \Ep{Z\sim\pi_Z}{\|M(Z)^\top M(Z)\|_{\infty\to 1}}.\]
Now define a rescaled version of the matrix,
\[\tilde{M}_{i,\bfa}(z) = \sqrt{\pi_i(z) \cdot e^{-\lambda} \cdot \prod_{j\geq 1} (e^{\lambda q_j}-1)^{a_j}} \cdot  \left( 1 - \frac{\one{a_i=1}}{1-e^{-\lambda q_i}}\right).\]
Since $\bfa\mapsto e^{-\lambda} \cdot \prod_{j\geq 1} (e^{\lambda q_j}-1)^{a_j}$ is the probability mass function of a distribution (i.e., of the distribution of $\bfa$ determined by the Poissonized missingness mechanism), as in the proof of Proposition~\ref{prop:example1} we have
\[\|M(z)^\top M(z)\|_{\infty\to 1}\leq \|\tilde{M}(z)^\top \tilde{M}(z)\|_2=\|\tilde{M}(z)\tilde{M}(z)^\top\|_2.\]
Finally, we calculate $\tilde{M}(z)\tilde{M}(z)^\top$. For any $i,k\geq 1$, we compute
\begin{align*}
    &(\tilde{M}(z)\tilde{M}(z)^\top)_{ik}
    =\sum_{\bfa}\tilde{M}_{i,\bfa}(z)\tilde{M}_{k,\bfa}(z)\\
    &=\sqrt{\pi_i(z)\pi_k(z)}\cdot \sum_{\bfa} e^{-\lambda} \cdot \prod_{j\geq 1} (e^{\lambda q_j}-1)^{a_j} \cdot \left( 1 - \frac{\one{a_i=1}}{1-e^{-\lambda q_i}}\right)\cdot \left( 1 - \frac{\one{a_k=1}}{1-e^{-\lambda q_k}}\right)\\
    &=\sqrt{\pi_i(z)\pi_k(z)}\cdot\Ep{\bfa}{\left( 1 - \frac{\one{a_i=1}}{1-e^{-\lambda q_i}}\right)\cdot \left( 1 - \frac{\one{a_k=1}}{1-e^{-\lambda q_k}}\right)}\\
    &=\sqrt{\pi_i(z)\pi_k(z)}\cdot\left(1 - \frac{\Pp{\bfa}{a_i=1}}{1-e^{-\lambda q_i}}- \frac{\Pp{\bfa}{a_k=1}}{1-e^{-\lambda q_k}}+ \frac{\Pp{\bfa}{a_i=a_k=1}}{(1-e^{-\lambda q_i})(1-e^{-\lambda q_k})}\right),
\end{align*}
where the last two lines are computed with respect to the distribution of $\bfa$ under the Poissonized missingness mechanism. Under this distribution, we have $\Pp{\bfa}{a_i=1} = 1-e^{-\lambda q_i}$ and $\Pp{\bfa}{a_k=1}=1-e^{-\lambda q_k}$, so the above simplifies to
\[(\tilde{M}(z)\tilde{M}(z)^\top)_{ik} = \sqrt{\pi_i(z)\pi_k(z)}\cdot\left(\frac{\Pp{\bfa}{a_i=a_k=1}}{(1-e^{-\lambda q_i})(1-e^{-\lambda q_k})} - 1\right).\]
If $i\neq k$, then $a_i\independent a_k$, and so $\Pp{\bfa}{a_i=a_k=1}=(1-e^{-\lambda q_i})(1-e^{-\lambda q_k})$, meaning that
\[(\tilde{M}(z)\tilde{M}(z)^\top)_{ik}=0 \textnormal{ for $i\neq k$}.\]
On the other hand, for the diagonal entries $i=k$, we have $\Pp{\bfa}{a_i=a_k=1}=\Pp{\bfa}{a_i=1} = 1-e^{-\lambda q_i}$, and so
\[(\tilde{M}(z)\tilde{M}(z)^\top)_{ii} = \pi_i(z)\cdot\left(\frac{1-e^{-\lambda q_i}}{(1-e^{-\lambda q_i})^2} - 1\right) = \frac{\pi_i(z)}{e^{\lambda q_i}-1}.\]
Therefore, $\tilde{M}(z)\tilde{M}(z)^\top$ is a diagonal matrix, with
\[\|\tilde{M}(z)\tilde{M}(z)^\top\|_2 = \sup_i (\tilde{M}(z)\tilde{M}(z)^\top)_{ii} = \sup_i \frac{\pi_i(z)}{e^{\lambda q_i}-1}\leq \frac{\sup_i \{\pi_i(z)/q_i\}}{\lambda},\]
since $e^{\lambda q_i}\geq 1 + \lambda q_i$. This completes the proof.
\end{proof}

\subsection{Additional technical results}

\begin{proof}[Proof of Proposition~\ref{prop:compare_to_dtv}]
    By Lemma~\ref{lem:operator_avg}, we have
    \[ \|T_{P,Q}\|_{L_\infty\to L_1} = \sup\left\{ \Ep{(Z,Z')\sim P}{\big(f_Q(Z)-f_Q(Z')\big)\cdot \big(g_Q(Z)-g_Q(Z')\big)} : \|f\|_{L_\infty},\|g\|_{L_\infty}\leq 1\right\},\]
    where $f_Q,g_Q$ are defined as in Lemma~\ref{lem:operator_avg}. For any function $f$ taking values in $[-1,1]$, we can write
    \[f_Q(Z)-f_Q(Z') = \Ep{Z''\sim Q(\cdot\mid Z)}{f(Z'')} - \Ep{Z''\sim Q(\cdot\mid Z')}{f(Z'')} \leq 2\dtv(Q(\cdot\mid Z),Q(\cdot\mid Z')),\]
    by definition of total variation distance. A similar bound holds for $g_Q$. Therefore,
    \[ \Ep{(Z,Z')\sim P}{\big(f_Q(Z)-f_Q(Z')\big)\cdot \big(g_Q(Z)-g_Q(Z')\big)}\leq \Ep{(Z,Z')\sim P}{4\Big(\dtv(Q(\cdot \mid Z),Q(\cdot\mid Z'))\Big)^2}\]
    holds for any $f,g$ with $\|f\|_{L_\infty},\|g\|_{L_\infty}\leq 1$, which completes the proof.
\end{proof}

\begin{lemma}\label{lem:sum_var_subset}
    Fix any integer $n\geq 1$. Let $X_1,\dots,X_n$ be independent and let $f(X_1,\dots,X_n)$ be any square-integrable function. For any subset $S\subseteq[n]$, define a random variable
    \[E_S = \EEst{f(X_1,\dots,X_n)}{X_S},\]
    the conditional expectation given $X_S = (X_i)_{i\in S}$. Then, for any $s\in[n]$,
    \[\frac{1}{\binom{n}{s}}\sum_{|S|=s} \textnormal{Var}(E_S) \leq \frac{s}{n}\cdot \textnormal{Var}(f(X_1,\dots,X_n)).\]
\end{lemma}
\noindent We can interpret the lemma as follows: if we condition on $X_S$ for a \emph{random} subset $S$ of size $s$, we expect to capture a proportional amount of the variability of the function, i.e., $\frac{s}{n}$ of the total variance.

\begin{proof}[Proof of Lemma~\ref{lem:sum_var_subset}]
    We will prove this by induction on $n$. For the base case $n=1$, we must have $s=1$ and so the claim holds trivially.

    Now fix some $n\geq 2$. If $s=n$, then again the result is trivial since the only possible subset is $S=[n]$. Now assume $s\leq n-1$. 
    
    Fix any $i\in[n]$ and define $E_{-i} = \EEst{Y}{X_{-i}}$ where $Y=f(X_1,\dots,X_n)$. Note that, if $S\subseteq[n]\setminus\{i\}$, then
    \[E_S = \EEst{Y}{X_S} = \EEst{\EEst{Y}{X_{-i}}}{X_S} = \EEst{E_{-i}}{X_S}.\]
    Now apply the result of the lemma, with $n-1$ in place of $n$, and with $E_{-i}$ in place of $Y$: we obtain
    \begin{equation}\label{eqn:induction_step__lem:sum_var_subset}\sum_{\substack{S\subseteq[n]\setminus\{i\}\\|S|=s}} \textnormal{Var}(E_S) \leq \binom{n-2}{s-1}\cdot \textnormal{Var}(E_{-i}).\end{equation}
    Next we sum over all $i$. For the right-hand side of~\eqref{eqn:induction_step__lem:sum_var_subset} we calculate
    \begin{align*}
        \sum_{i=1}^n \textnormal{Var}(E_{-i}) 
        &=\sum_{i=1}^n \left(\textnormal{Var}(Y) - \EE{\textnormal{Var}(Y\mid X_{-i})}\right)\textnormal{ by the Law of Total Variance}\\
        &=n\textnormal{Var}(Y) - \sum_{i=1}^n\EE{\textnormal{Var}(Y\mid X_{-i})}\\
        &\leq (n-1)\textnormal{Var}(Y),
    \end{align*}
    where the last step holds by the Efron--Stein inequality
   \citep[Theorem 3.1]{boucheron2013concentration}, which tells us that
\[
\textnormal{Var}(f(X_1,\dots,X_n))\leq \sum_{i=1}^n\EE{\textnormal{Var}(f(X_1,\dots,X_n) \mid X_{-i})}.
\]
    On the other hand, for the left-hand side of~\eqref{eqn:induction_step__lem:sum_var_subset}, we have
    \[\sum_{i=1}^n \sum_{\substack{S\subseteq[n]\setminus\{i\}\\|S|=s}} \textnormal{Var}(E_S) = \sum_{\substack{S\subseteq[n]\\|S|=s}}\left(\sum_{i=1}^n\one{i\not\in S}\right)\textnormal{Var}(E_S) = \sum_{\substack{S\subseteq[n]\\|S|=s}}\textnormal{Var}(E_S) \cdot (n-s). \]
    Returning to~\eqref{eqn:induction_step__lem:sum_var_subset}, then, we have shown that
    \[\sum_{\substack{S\subseteq[n]\\|S|=s}}\textnormal{Var}(E_S) \cdot (n-s) \leq \binom{n-2}{s-1}\cdot (n-1)\textnormal{Var}(Y).\]
    After simplifying, this proves the desired claim.
    \end{proof}

\section{Extension to Hilbert space-valued outputs}
\label{sec:HSExtension}

In this section, we extend our main stability guarantee to avoid requiring that the algorithm~$\alg$ returns real-valued (i.e., one-dimensional) outputs. In particular, let $\cH$ be a separable Hilbert space---this includes examples such as $\R^d$ (for multivariate outputs) or $L_2(\R)$ (i.e., square-integrable functions $f:\R\to\R$, for function-valued outputs). Let $\cW\subseteq\cH$ be a convex and closed subset, and define its radius
\[\textnormal{rad}_\cH(\cW) = \inf_{w\in\cW}\sup_{w'\in\cW}\|w-w'\|_\cH.\]
We will assume that $\alg$ returns outputs in $\cW$, i.e., we can write the algorithm as a map $\alg:\cZ\to\cW$. Our goal is to examine its stability,
\[\beta^2_P(\alg) = \Ep{(Z,Z')\sim P}{\|\alg(Z)-\alg(Z')\|^2_\cH},\]
where now we measure the change in the output using the Hilbert space norm $\|\cdot\|_\cH$.

The work of \citet{soloff2024stability} establishes the stability properties of bagging in the setting of a Hilbert-space-valued output. We now show an analogous result working in our general framework:
\begin{theorem}[Extension of Theorem~\ref{thm:main} to Hilbert-space-valued output]\label{thm:main_hilbert}
    Let $\cH$ be a separable Hilbert space and let $\cW\subseteq\cH$ be a convex and closed subset with $\textnormal{rad}_\cH(\cW)<\infty$. For any $\alg:\cZ\to\cW$, it holds that
    \[\beta^2_P(\alg_Q) \leq \textnormal{rad}_\cH(\cW)^2 \cdot \|T_{P,Q}\|_{L_\infty(\cH)\to L_1(\cH)}.\]
\end{theorem}
\noindent Essentially, the result of Theorem~\ref{thm:main_hilbert} is a direct extension of the real-valued-case: in particular, we recover the result of Theorem~\ref{thm:main} as a special case by taking $\cH=\R$ and $\cW=[0,1]$, since in this case we have $\textnormal{rad}_\cH(\cW) = 1/2$.

In the result above, the operator $T_{P,Q}$ is defined exactly as before, but its norm is now defined with respect to $\cH$-valued functions. Specifically, let $f:\cZ\to\cH$ be bounded and measurable. We define
\[
[T_{P,Q}f](z) = \int_{\cZ} K_{P,Q}(z,z') f(z')\;\mathsf{d}\mu(z'),
\]
as before, where now the returned function $T_{P,Q}f$ is also $\cH$-valued.
Then write
\[\|T_{P,Q}\|_{L_\infty(\cH)\to L_1(\cH)} = \sup\left\{\|T_{P,Q}f\|_{L_1(\cH)} : \|f\|_{L_\infty(\cH)} \leq 1\right\},\]
where for any measurable $f:\cZ\to\cH$ we define $\|f\|_{L_1(\cH)} = \int_{\cZ} \|f(z)\|_\cH\;\mathsf{d}\mu(z)$, and $\|f\|_{L_\infty(\cH)} = \sup_{z \in \cZ} \|f(z)\|_\cH$.  We show in Proposition~\ref{Prop:Grothendieck} that by an extension of Grothendieck's inequality \citep{grothendieck1953sur}, 
\[\|T_{P,Q}\|_{L_\infty(\cH)\to L_1(\cH)} \leq C_{\mathrm{G}} \|T_{P,Q}\|_{L_\infty\to L_1},\]
where $C_{\mathrm{G}}$ denotes Grothendieck's constant, known to satisfy $1.67 < C_{\mathrm{G}} < 1.79$. In particular, this means that the stability bound shown in Theorem~\ref{thm:main_hilbert} is (up to this mild constant) no worse than the bound for real-valued algorithms, given in Theorem~\ref{thm:main}.
\begin{proof}[Proof of Theorem~\ref{thm:main_hilbert}]
Let $w^*\in\cW$ attain the infimum, $\inf_{w\in\cW}\sup_{w'\in\cW}\|w-w'\|_\cH$, so that
\[\textnormal{rad}_\cH(\cW) = \sup_{w\in\cW}\|w-w^*\|_\cH.\]
Define
 $f(z) = \alg(z) - w^*$. 
Note that $\alg_Q(z) = f_Q(z)+w^* $ for all $z$, by construction. 

From this point on the proof is identical to that of Theorem~\ref{thm:main}, except that we work in the Hilbert space $\cH$. We calculate
\begin{align*}
    \beta^2_P(\alg_Q)
    &=\Ep{P}{\left\|\alg_Q(Z_0)-\alg_Q(Z_1)\right\|_\cH^2}\\
    &=\Ep{P}{\left\langle f_Q(Z_0)-f_Q(Z_1), f_Q(Z_0)-f_Q(Z_1)\right\rangle}\\
    &= \int_{\cZ} \left\langle f(z),  [T_{P,Q}f](z)\right\rangle \;\mathsf{d}\mu(z)\textnormal{ by Lemma~\ref{lem:operator_avg_hilbert} (applied with $f=g$)}\\
    &\leq \|f\|_{L_\infty(\cH)}\|T_{P,Q}f\|_{L_1(\cH)}
    \leq \|f\|_{L_\infty(\cH)}^2\|T_{P,Q}\|_{L_\infty(\cH)\to L_1(\cH)}\\
    &\leq \textnormal{rad}_\cH(\cW)^2\cdot \|T_{P,Q}\|_{L_\infty(\cH)\to L_1(\cH)},
\end{align*}
where the last step holds since $f$ takes values in $\{w - w^* : w\in\cW\}$ and therefore $\|f\|_{L_\infty(\cH)}\leq \textnormal{rad}_\cH(\cW)$, by construction.
\end{proof}

We can also take a relaxation of this result, as for the real-valued case:
\begin{corollary}\label{cor:main_hilbert_L2}
    Let $\psi$ be any density with respect to the base measure $\mu$, with $\psi(z)>0$ for all $z$.
    Then
    for any $\alg:\cZ\to\cW$, it holds that
    \[\beta^2_P(\alg_Q) \leq \textnormal{rad}_\cH(\cW)^2\cdot \|\tilde{T}^\psi_{P,Q}\|_{L_2\to L_2},\]
    where we define the operator $\tilde{T}^\psi_{P,Q}$ as
    \[[\tilde{T}^\psi_{P,Q}f](z) = \int_\cZ \frac{K_{P,Q}(z,z')}{\sqrt{\psi(z)\psi(z')}} \cdot f(z')\;\mathsf{d}\mu(z'). \]
\end{corollary}
\noindent Note that, once we relax to the $L_2\to L_2$ norm, there is no price to pay for extending from real-valued to Hilbert-valued functions: the quantity $\|\tilde{T}^\psi_{P,Q}\|_{L_2\to L_2}$ appearing in the bound is the same norm as for the real-valued case.
\begin{proof}[Proof of Corollary~\ref{cor:main_hilbert_L2}]
    Exactly as in the proof of Corollary~\ref{cor:main_L2}, we bound
    \[\|T_{P,Q}\|_{L_\infty(\cH)\to L_1(\cH)} \leq \|\tilde{T}^\psi_{P,Q}\|_{L_2(\cH)\to L_2(\cH)}.\]
    And, by properties of the Hilbert norm, we also have
    \[\|\tilde{T}^\psi_{P,Q}\|_{L_2(\cH)\to L_2(\cH)} = \|\tilde{T}^\psi_{P,Q}\|_{L_2\to L_2}.\]
\end{proof}

\begin{lemma}[Extension of Lemma~\ref{lem:operator_avg} to Hilbert-space-valued output]\label{lem:operator_avg_hilbert}
    Let $f,g:\cZ\to\cH$ be bounded functions. Define their ensembled versions
    \[f_Q(z) = \Ep{Z\sim Q(\cdot\mid z)}{f(Z)}, \quad g_Q(z) = \Ep{Z\sim Q(\cdot\mid z)}{g(Z)}.\]
    Then
    \[\int_{\cZ} \left\langle f(z),[T_{P,Q}g](z)\right\rangle \;\mathsf{d}\mu(z) = \Ep{P}{\left\langle f_Q(Z_0)-f_Q(Z_1),g_Q(Z_0)-g_Q(Z_1)\right\rangle}.\]
\end{lemma}
\begin{proof}[Proof of Lemma~\ref{lem:operator_avg_hilbert}] The proof is identical to that of Lemma~\ref{lem:operator_avg}, except that we now work with inner products in the separable Hilbert space $\cH$. We calculate
    \begin{align*}
    &\int_{\cZ} \left\langle f(z),[T_{P,Q}g](z)\right\rangle\;\mathsf{d}\mu(z)\\
    &=\int_{\cZ} \left\langle f(z) ,\int_{\cZ} K_{P,Q}(z,z')\cdot g(z')\;\mathsf{d}\mu(z')\right\rangle\;\mathsf{d}\mu(z)\\
    &=\int_{\cZ} \int_{\cZ} \Ep{P}{\big(Q(z\mid Z_0) - Q(z\mid Z_1)\big)\cdot \big(Q(z'\mid Z_0) - Q(z'\mid Z_1)\big)}\cdot \langle f(z),g(z')\rangle\;\mathsf{d}\mu(z')\;\mathsf{d}\mu(z)\\
    &=\Ep{P}{\int_{\cZ} \int_{\cZ} \big(Q(z\mid Z_0) - Q(z\mid Z_1)\big)\cdot \big(Q(z'\mid Z_0) - Q(z'\mid Z_1)\big)\cdot \langle f(z),g(z')\rangle\;\mathsf{d}\mu(z')\;\mathsf{d}\mu(z)}\\
    &=\Ep{P}{\left\langle \int_{\cZ} \big(Q(z\mid Z_0) - Q(z\mid Z_1)\big) f(z)\;\mathsf{d}\mu(z) , \int_{\cZ} \big(Q(z'\mid Z_0) - Q(z'\mid Z_1)\big) g(z')\;\mathsf{d}\mu(z')\right\rangle }\\
    &=\Ep{P}{\left\langle f_Q(Z_0)-f_Q(Z_1),g_Q(Z_0)-g_Q(Z_1)\right\rangle},
\end{align*}
where the second step holds by definition of $K_{P,Q}$, the third step holds by Fubini's theorem (applied in the separable Hilbert space $\cH$), and the last step holds by definition of $f_Q,g_Q$.
\end{proof}

The following lemma concerns an extension of the simplest, finite-dimensional form of Grothendieck's inequality to separable Hilbert spaces; since the proof is self-contained and direct, we provide it for the reader's convenience.   
\begin{proposition}
\label{Prop:Grothendieck}
We have
\[
\|T_{P,Q}\|_{L_\infty(\mathcal{H}) \rightarrow L_1(\mathcal{H})} \leq C_{\mathrm{G}} \|T_{P,Q}\|_{L_\infty \rightarrow L_1}.
\]
\end{proposition}
\begin{proof}
First define functions $F,G:\mathcal{Z} \rightarrow \mathcal{H}$ that take only finitely many values via
\[
F(z') = \sum_{j=1}^n u_j \mathbbm{1}_{A_j}(z'), \quad G(z) = \sum_{i=1}^n v_i \mathbbm{1}_{B_i}(z),
\]
where $A_1,\ldots,A_n$ are measurable, pairwise disjoint subsets of $\mathcal{Z}$, as are $B_1,\ldots,B_n$, and both $\max_{j \in [n]} \|u_j\|_{\mathcal{H}} \leq 1$ and $\max_{i \in [n]} \|v_i\|_{\mathcal{H}} \leq 1$.  Thus $\|F\|_{L_\infty(\mathcal{H})}, \|G\|_{L_\infty(\mathcal{H})} \leq 1$.  Define $A = (a_{ij}) \in \mathbb{R}^{n \times n}$ by
\[
a_{ij} = \int_{B_i} \int_{A_j} K_{P,Q}(z,z') \; \mathsf{d}\mu(z') \; \mathsf{d}\mu(z).
\]
Now let $s_1,\ldots,s_n,t_1,\ldots,t_n \in [-1,1]$, and define simple functions $f,g:\mathcal{Z} \rightarrow \mathbb{R}$ by
\[
f(z') = \sum_{j=1}^n s_j \mathbbm{1}_{A_j}(z'), \quad g(z) = \sum_{i=1}^n t_i \mathbbm{1}_{B_i}(z).
\]
Then $\|f\|_\infty, \|g\|_\infty \leq 1$, so 
\[
\biggl|\sum_{i=1}^n \sum_{j=1}^n a_{ij} s_jt_i\biggr| = \biggl|\int_{\mathcal{Z}} \int_{\mathcal{Z}} K_{P,Q}(z,z') f(z')g(z) \;\mathsf{d}\mu(z') \; \mathsf{d}\mu(z)\biggr| \leq \|T_{P,Q}\|_{L_\infty \rightarrow L_1}.
\]
It follows by the finite-dimensional Grothendieck inequality \citep[e.g.][Theorem~1.1]{pisier2012}, that
\[
\biggl|\int_{\mathcal{Z}} \int_{\mathcal{Z}} K_{P,Q}(z,z')\langle F(z'), G(z) \rangle_{\mathcal{H}} \; \mathsf{d}\mu(z') \; \mathsf{d}\mu(z) \biggr| = \biggl|\sum_{i=1}^n \sum_{j=1}^n a_{ij} \langle u_j,v_i \rangle_{\mathcal{H}}\biggr| \leq C_{\mathrm{G}} \|T_{P,Q}\|_{L_\infty \rightarrow L_1}.
\]
Now let $F:\mathcal{Z} \rightarrow \mathcal{H}$ be measurable with $\|F\|_{L_\infty(\mathcal{H})} \leq 1$.  Let $(e_\ell)_{\ell \in \mathbb{N}}$ denote an orthonormal basis for $\mathcal{H}$, so that we may write
\[
F(z) = \sum_{\ell=1}^\infty \lambda_\ell(z)e_\ell,
\]
where $\sup_{z \in \mathcal{Z}} \sum_{\ell=1}^\infty \lambda_\ell^2(z) \leq 1$.  
Now let $\mathcal{H}_n = \mathrm{span}(e_1,\ldots,e_n)$, let $B_{\mathcal{H}_n} = \{x \in \mathcal{H}_n:\|x\|_{\mathcal{H}} \leq 1\}$, and let $\{u_{n,1},\ldots,u_{n,N_n}\} \subseteq \mathcal{H}_n$ denote a $(1/n)$-net for $B_{\mathcal{H}_n}$.  Thus, given any $x \in B_{\mathcal{H}_n}$, we can find $r \in [N_n]$ such that $\|x - u_{n,r}\|_{\mathcal{H}} \leq 1/n$.  Define 
\[
E_{n,1} = \biggl\{z \in \mathcal{Z}:\biggl\|\sum_{\ell=1}^n \lambda_\ell(z) e_\ell - u_{n,1}\biggr\|_{\mathcal{H}} \leq \frac{1}{n} \biggr\}
\]
and
\[
E_{n,r} = \biggl\{z \in \mathcal{Z}:\biggl\|\sum_{\ell=1}^n \lambda_\ell(z) e_\ell - u_{n,r}\biggr\|_{\mathcal{H}} \leq \frac{1}{n} \biggr\} \setminus \bigcup_{s=1}^{r-1} E_{n,s}
\]
for $r \in [N_n] \setminus \{1\}$.  The sets $(E_{n,r})_{r \in [N_n]}$ form a measurable partition of $\mathcal{Z}$, because we have $\sup_{z \in \mathcal{Z}} \bigl\|\sum_{\ell=1}^n \lambda_\ell(z)e_\ell\bigr\|_{\mathcal{H}}^2 = \sup_{z \in \mathcal{Z}} \sum_{\ell=1}^n \lambda_\ell^2(z) \leq 1$.  Now define a measurable function $F_n:\mathcal{Z} \rightarrow \mathcal{H}$ taking only finitely many values by
\[
F_n(z) = \sum_{r=1}^{N_n} u_{n,r} \mathbbm{1}_{E_{n,r}}(z). 
\]
Then $\sup_{z \in \mathcal{Z}} \|F_n(z)\|_{\mathcal{H}} \leq 1$ and for every $z \in \mathcal{Z}$,
\begin{align*}
\|F_n(z) - F(z)\|_{\mathcal{H}} &\leq \biggl\|\sum_{r=1}^{N_n} u_{n,r} \mathbbm{1}_{E_{n,r}}(z) - \sum_{\ell=1}^n \lambda_\ell(z) e_\ell\biggr\|_{\mathcal{H}} + \biggl\|\sum_{\ell=1}^n \lambda_\ell(z) e_\ell - F(z)\biggr\|_{\mathcal{H}} \\
&\leq \frac{1}{n} + \biggl(\sum_{\ell=n+1}^\infty \lambda_\ell^2(z)\biggr)^{1/2} \rightarrow 0
\end{align*}
as $n \rightarrow \infty$.

Since $|\langle F_n(z'),F_n(z) \rangle_{\mathcal{H}}| \leq  \|F_n(z')\|_{\mathcal{H}}\|F_n(z)\|_{\mathcal{H}} \leq 1$, we may apply dominated convergence to deduce that
\begin{align*}
\biggl|\int_{\mathcal{Z}} \int_{\mathcal{Z}} K_{P,Q}(z,z')\langle F(z'), F(z)\rangle_{\mathcal{H}} \; &\mathsf{d}\mu(z') \; \mathsf{d}\mu(z)\biggr| \\
&= \lim_{n \rightarrow \infty} \biggl|\int_{\mathcal{Z}} \int_{\mathcal{Z}} K_{P,Q}(z,z')\langle F_n(z'), F_n(z)\rangle_{\mathcal{H}} \; \mathsf{d}\mu(z') \; \mathsf{d}\mu(z)\biggr| \\
&\leq C_{\mathrm{G}} \|T_{P,Q}\|_{L_\infty \rightarrow L_1}.
\end{align*}
We conclude that
\begin{align*}
\|T_{P,Q}\|_{L_\infty(\mathcal{H}) \rightarrow L_1(\mathcal{H})} &= \sup\biggl\{\int_{\mathcal{Z}} K_{P,Q}(z,z') \langle F(z'),F(z)\rangle_{\mathcal{H}} \; \mathsf{d}\mu(z') \; \mathsf{d}\mu(z):\|F\|_{L_\infty(\mathcal{H})} \leq 1\biggr\} \\
&\leq C_{\mathrm{G}} \|T_{P,Q}\|_{L_\infty \rightarrow L_1},
\end{align*}
as required.
\end{proof}

\end{document}